\documentclass[12pt]{article}

\def\edo{\end{document}}

\usepackage{amsmath, amsthm, amscd, amsfonts, amssymb}

\newtheorem{theorem}{Theorem}[section]
\newtheorem{proposition}[theorem]{Proposition}

\newtheorem{lemma}[theorem]{Lemma}

\newtheorem{remark}[theorem]{Remark}

\def\divv{{\rm div }}
\def\rrd{{\mathbb{R}^d}}
\def\ae{\mbox{ a.e. }}
\def\la{\left\{}
\def\ra{\right\}}

\def\calf{{\mathcal{F}}}

\def\calx{{\mathcal{X}}}

\def\call{{\mathcal{L}}}
\def\calu{{\mathcal{U}}}

\def\calp{{\mathcal{P}}}

\def\vsp{\vspace*{1,5mm}\\ }
\def\bk{\bigskip }
\def\mk{\medskip }
\def\sk{\smallskip }
\def\n{\noindent }
\def\dd{\displaystyle}

\def\barr{\begin{array}}
\def\earr{\end{array}}

\def\bit{\begin{itemize}}
\def\eit{\end{itemize}}

\def\FP{Fokker--Planck}

\def\r{{\rho}}
\def\1{^{-1}}

\def\E{{\mathbb{E}}}
\def\nn{{\mathbb{N}}}
\def\PP{{\mathbb{P}}}
\def\rr{{\mathbb{R}}}
\def\9{{\infty}}
\def\lbb{{\lambda}}
\def\wt{\widetilde}
\def\ov{\overline}
\def\vf{{\varphi}}
\def\oo{{\omega}}
\def\ooo{{\Omega}}
\def\pp{{\partial}}
\def\vp{{\varepsilon}}

\def\ff{\forall }
\def\({\left(}
\def\){\right)}
\def\<{\left<}
\def\>{\right>}

\title{\bf Mean field system:\\ the optimal control based approach} 
\author{{\bf Viorel Barbu}\thanks{Octav Mayer Institute of Mathematics of  Romanian Academy  and Al.I. Cuza University,   Ia\c si, Romania.  Email: vbarbu41@gmail.com}\ \thanks{{\it Funding:} This work was supported by a grant of Ministry of Research, Inovation and Digitalization, CNCS-UEFISCDI project PN-III-P4-PCE-2021-0006, within PNCDI III.}}
\date{}

\begin{document}
\maketitle
\begin{abstract}
\n The mean field game system arises in the construction of  Nash type strategies for large population dynamics games. It is a  nonlinear parabolic system in  $(0,\9)\times\rrd$ which consists of a nonlinear forward \FP\ equation coupled with a backward Hamilton--Jacobi equation. In this work, the mean field game system is treated by interpreting it as the Euler--Lagrange system cor\-res\-pon\-ding to a Bolza control problem governed by a li\-near \FP\ equation with the controller in the drift term. One obtains in this way the existence and uniqueness of strong solutions  in appropriate Sobolev spaces. This variational approach permits to treat more ge\-ne\-ral mean field game systems than that studied in literature and, in particular, that with multivalued nonlinearities in the \FP\ equation and discontinuous coupling functions.\sk\\
{\bf Keywords:} mean field game, \FP\ equation, stochastic differential equation, optimal control.\\ 
{\bf MSC Codes:} 60H15, 47H05, 47J05.\\
\end{abstract}

\section{Introduction}\label{s1}
Here we are concerned with the partial differential system (the mean field game system) \bk
 
\begin{equation}\label{e1.1}
\barr{c}
\dd\frac{\pp\r}{\pp t}-\nu\Delta\r+{\rm div}(\r H_q(t,x,\nabla p))=0,\ (t,x)\in(0,T)\times\rrd,\vsp 
\dd\frac{\pp p}{\pp t}+\nu\Delta p+H(t,x,\nabla p)=F(t,x,\r),\ (t,x)\in(0,T)\times\rrd,\vsp 
\r(0,x)=\r_0(x),\ p(T,x)+G(x,\r(T,x))=0,\ x\in\rrd,\earr  
\end{equation}where $\nu>0$, $H=H(t,x,q)$ is convex and continuous in $q\in\rrd,$ measurable in $(t,x)\in Q_T=(0,T)\times\rrd$ and $H_q$ is the subdifferential of the function \mbox{$q\to H(t,x,q)$.} The coupling functions $F\equiv F(t,x,r)=\pp_rg(t,x,r)$ and $G=G(x,r)$ are monotonically increasing in $r$ (eventually multivalued) and measurable in $(t,x)\in Q_T$.  This system is relevant in mean field games theory initiated by J.M.~Lasry and P.L.~Lions~\cite{9} and M.~Huang, P.E.~Caines and R.P.~Malham\'e~\cite{8}. (See also \cite{5}, \cite{6}, \cite{6a}, \cite{7a}, \cite{10} for other significant contributions.) 
Let us briefly present the standard mean field game model defined by the stochastic optimal control problem (see \cite{4})
\begin{equation}\label{e1.2}
\underset u{\rm Min} \Big\{\E\!\!\dd\int^T_0\!\!(L(t,X_t,u(t,X_t)){+}F(t,X_t,\r(t)))dt 
{+}\E G(X_T,\r(T));\, u\in\calu\Big\} 
\end{equation}
where the process $X_t$ is the solution to the stochastic differential equation
\begin{equation}\label{e1.3}
\barr{l}
dX_t=u(t,X_t)dt+\sqrt{2\nu}\,dW_t,\ t\in(0,T),\vsp 
X(0)=X_0,\earr\end{equation}
in a probability space $(\calf,\calf_t,\PP)$ with the $d$-dimensional Brownian motion $W_t$. Here $\r=\r(t,x)$ is the probability density of the law $\call_{X_t}$ of the process $X_t$ with respect to the Lebesgue measure and the control constraint set $\calu$ is a closed subset in the space of all real-valued measurable functions \mbox{$u:Q_T\to\rrd$.} 
The controlled SDE \eqref{e1.3} describes the dynamics of an individual player which aims to minimizing the cost functional \eqref{e1.2} and get so its own optimal strtegy $u=u(t,X_t)$. Let us assume first that $\rho(t)\equiv\rho(t,x)$ is fixed in the class $\calp$ of probability densities in $\rrd$ (see \eqref{e1.12a} below). Hence, \eqref{e1.2}--\eqref{e1.3} can be viewed as a stochastic optimal control problem which depends on the parameter $\r$ with the corresponding optimal value function $\psi=\psi^\r$
\begin{equation}\label{e1.4}
\barr{r}
\psi(t,x)=\dd\inf_u\Big\{\E\dd\int^T_t(L(s,X_s,u_s(s,X_s))+F(s,X_s,\r(s))ds\\
+\E G(X_T,\r(T)));\ u\in\calu\Big\}\earr\end{equation}
which is a generalized (eventual viscosity) solution to the Hamilton--Jacobi equation (see, e.g., \cite{7})
\begin{equation}\label{e1.5}
\barr{ll}
\dd\frac{\pp\psi}{\pp t}+\nu\Delta\psi-H(t,x,-\nabla\psi)+F(t,x,\r)=0&\mbox{ in }Q_T,\vsp 
\psi(T,x)=G(x,\r(T,x))&\mbox{ in }\rrd.\earr\end{equation}
An optimal controller $u=u^\r$  for \eqref{e1.4}, assuming that it exists, is formally given~by
\begin{equation}\label{e1.6}
u^\r(t,x)=H_q(t,x,-\nabla\psi(t,x)),\ (t,x)\in Q_T,\end{equation}
where, $H:Q_T\times\rr\to\rr$ is the Hamiltonian corresponding to the Lagrangean function $L$, that is,
\begin{equation}\label{e1.7}
H(t,x,q)=\sup\{q\cdot v-L(t,x,v);\,v\in\rrd\},\ (t,x)\in Q_T,\ u\in\rrd.\end{equation}
Moreover, the probability density $\wt\r=\wt\r^\r$ of the corresponding optimal process $\wt X=\wt X^\r$  is a distributional solution to the \FP\ equation
\begin{equation}\label{e1.8}
\barr{ll}
\dd\frac{\pp\wt\r}{\pp t}-\nu\Delta\wt\r+\divv(\wt\r H_q(t,x,-\nabla\psi))=0&\mbox{in }Q_T,\vsp 
\wt\r(0,x)=\r_0(x),&x\in\rrd,\earr\end{equation}
where $\r_0$ is the density probability of $X_0$. 
More precisely, the equivalence of equations \eqref{e1.3} and \eqref{e1.8} via the formula $\wt\rho\,dx=\call_{\wt x_t}$ is between the martingale solutions (or probability weak solutions) $\wt X^\rho$ to \eqref{e1.3} and the distributional solutions $\wt\rho$ to \eqref{e1.8} which are weakly $L^1(\rrd)$-continuous in $t$. (See, e.g., \cite{2}, \cite{3a}.) Hence, for each fixed $\rho$ each  player has an optimal strategy $(\wt u,\wt\rho)$ which is defined by system \eqref{e1.5}, \eqref{e1.6}, \eqref{e1.8}. The mean field strategy is to take the parameter $\rho$ as an  optimal distribution anti\-ci\-pa\-ted by all players and so $\rho=\wt\rho$. Then, system \eqref{e1.5}--\eqref{e1.8} for $(p=-\psi,\ \wt\rho=\rho)$ reduces to \eqref{e1.1} and the corresponding stra\-tegy for the mean field game \eqref{e1.2}--\eqref{e1.5} is a Nash-type equilibrium which, as mentioned above, amounts to saying that each player anticipates the other players choices  options.  

The existence and uniqueness of classical and weak solutions for \eqref{e1.1} were obtained in the pioneering work \cite{9} of Lasry and Lions and later on by A.~Poretta \cite{10} (see also \cite{4}, \cite{5}, \cite{6}). Most of these results are given on $(0,T)\times \pi^N$, where $\pi^d$ is the $d$-dimensional torus and for smooth coupling functions $H$, $F$ and~$G$.  The method used in the above works is based on fixed point theory combined with sharp existence results for linear parabolic equations. Here, we shall use a different approach to  the well-possedness of system \eqref{e1.1} by viewing it, and this observation arises first in Lasry and Lions work \cite{9}, as the optimality system corresponding to a deterministic optimal control problem governed by  the \FP\ equation. Namely,
\begin{equation}\label{e1.5a}
\underset u{\rm Min}\Big\{\!\dd\int_{Q_T}\!\!(L(t,x,u)\r
{+}g(t,x,\r))dtdx{+}\!\int_\rrd\!\! g_0(x,\r(T,x))dx,u\in L^1(Q_T)\!\Big\}
\end{equation}
\begin{equation}\label{e1.6a}
\barr{l}
\dd\frac{\pp\r}{\pp t}-\nu\Delta\r+\divv(u\r)=0\mbox{ in }Q_T,\vsp 
\r(0,x)=\r_0(x),\earr\end{equation}
where, $L$ is the Lagrangian function
\begin{equation}\label{e1.7a}
L(t,x,u)\equiv\sup\{q\cdot u-H(t,x,q)\}\end{equation}
and the functions $g:Q_T\times\rr\to\rr$, $g_0:\rrd\times\rr\to\rr$ are defined by
\begin{equation}\label{e1.8a}
\pp_r g(t,x,r)\equiv F(t,x,r);\ \pp_r g_0(x,r)\equiv G(t,x,r),\ \ff(t,x)\in Q_T;\ r\in\rr.\end{equation}
(One assumes here that $r\to F(\cdot,r)$ is a potential function on $\rrd$.) The main result of this work, Theorem \ref{t1}, amounts to saying that under appropriate conditions on $\r_0,$ $L$ and $g,$ $g_0$ to be made precise later on, the optimal control problem \eqref{e1.5a} has a  solution $(u,\r)$, where $u\in H_q(t,x,\nabla p)$, a.e. in $Q_T$,  and $(\r,p)$, where $p$ is the solution to the costate equation asso\-cia\-ted with \eqref{e1.6a}, is a solution to the mean field system \eqref{e1.1}. In particular, one obtains for $\r_0\in L^1(\rrd)$ such that $\r_0>0$, a.e. on $\rrd$ and $\log\r_0\in L^1_{\rm loc}$, the existence (and uniqueness) of a solution $(\r,p)$ to the mean field system \eqref{e1.1}.

These results can be compared most closely with that obtained by A.~Porretta \cite{10} (see also \cite{10a}), but the method involved here is different and the existence is obtained under more general conditions on the Hamiltonian $H$ and the functions $F$ and $G$ which, as seen later on, might be discontinuous. Compared with \cite{10}, one assumes here that the Hamiltonian function $H=H(t,x,q)$ is sublinear in $q$, which might seem restrictive but, taking into account the mean field game \eqref{e1.2}--\eqref{e1.3}, it is a natural one because this means that the control constraint set $\calu$ is of the form $\{u\in L^\9(Q_T)_;\,|u(t,x)|\le M$, a.e. $(t,x)\in Q_T\}$. 

As in other works on the mean field game system \eqref{e1.1}, the functions $F=F(t,x,q)$ and $G=G(t,x,r)$ are assumed monotonically nondecreasing in $q$ and $r$, respectively, but in this work they could be multivalued as well. (One should assume, however, that $F=F(\cdot,q)$ is a potential function,  that is, it is the subdifferential $\pp_q g(\cdot,q)$ of a continuous, convex function $q\to g(\cdot,q)$, a condition which in 1-$D$ is always satisfied for continuous and monotonically nondecreasing functions.) It should be emphasized  that taking into account the significance of system \eqref{e1.1} for the mean field game \eqref{e1.2}--\eqref{e1.3} the natural domain for it is $(0,T)\times\rrd=Q_T$. Of course, on bounded domains with Dirichlet boundary conditions the treatment is simpler but the first equation in \eqref{e1.1} cannot be interpreted as a \FP\ equation corresponding to a SDE of the form \eqref{e1.6a}. As regards the initial value $\rho_0$, taken into account its physical significance, it should be taking $L^1(\rrd)$ and more precisely in $\calp$.  

The variational approach used in this work involves two steps:\\ 1$^\circ$ the existence of a solution to the optimal control problem \eqref{e1.5a}--\eqref{e1.6a};\\ 2$^\circ$ the derivation of the maximum principle for this optimal control problem.\\ Both steps are quite delicate because the control system  \eqref{e1.6a} is bilinear in $(\r,u)$.  However, the main advantage of this approach is that, as usually happens in variational theory of PDEs, it provides smooth solutions (in Sobolev spaces) under relative weaker assumptions.

\bk\n{\bf Notation.} $L^p(\rrd),\ 1\le p\le\9$, also denoted $L^p$, is the space of Lebesgue measurable and $p$-integrable functions on $\rrd$ with the norm denoted $|\cdot|_p$. The scalar product in $L^2$ is denoted by $(\cdot,\cdot)_2$ and the corresponding norm by $|\cdot|_2$. We shall denote by $(L^p)^d$ the cor\-res\-pon\-ding space of vectorial $d$-dimensional functions. For $k\ge1$ and $p\in[1,\9]$, we denote by $W^{k,p}$   standard Sobolev spaces $W^{k,p}(\rrd)$ on $\rrd$. By $L^p_{\rm loc}$ and $W^{k,p}_{\rm loc}$ denote the corresponding local Sobolev spaces on $\rrd$. The same notations are used when $\rrd$ is replaced by $Q_T=(0,T)\times\rrd$. In the following, we shall simply write, for $1\le p\le\9$, $(L^p(Q_T))^d$ instead of $L^p(0,T;(L^p)^d)$. We set $H^1=W^{1,2}$ and denote by $H\1$ the dual space of $H^1$. Given a Banach space $\calx$, we  denote by $C([0,T];\calx)$ and $L^p(0,T;\calx)$ the spaces of $\calx$-valued continuous functions on $[0,T]$ and  $\calx$-valued  $L^p$-Bochner integrable  functions on $(0,T)$, respectively. We denote by 
 $W^{1,p}([0,T];\calx)$ the Sobolev space $\{y\in L^p(0,T;\calx);$  $\frac{dy}{dt}\in L^p(0,T;\calx)\}$ where $\frac{dy}{dt}$ is taken in the sense of $\calx$-valued vectorial distributions. We denote also by $W^{-1,p'}$, $\frac 1{p'}=1-\frac1p$ the dual of the space $W^{1,p}$ and we~set
\begin{equation}\label{e1.12a}
	\calp=\left\{\r\in L^1(\rrd);\ \r\ge0,\mbox{ a.e. in }\rrd,\ \int_\rrd\r\,dx=L\right\}.\end{equation}
If $f:\rrd\to\ ]-\9,+\9]$ is a {\it convex} and  {\it lower-semicontinuous} function, then  $\pp f(u)\subset\rrd$,   also denoted $f_u(u)$, is the  {\it subgradient} of $f$ in $u$, that is,
\newpage
\begin{equation}\label{e1.9}
\pp f(u)=\{\eta\in\rrd;\,f(u)\le f(v)+\eta\cdot(u-v),\ \ff v\in\rrd\}.\end{equation}
We denote by $f^*:\rrd\to\ ]-\9,+\9]$ the conjugate of $f$, that is,
\begin{equation}\label{e1.10}
f^*(v)=\sup\{u\cdot v-f(u);\ u\in\rrd\},\ v\in\rrd\},\end{equation}
and recall(see, e.g., \cite{1}, p.~84) that $\pp f^*=(\pp f)\1$ and
\begin{equation}\label{e1.11}
f(u)+f^*(\eta)=u\cdot\eta,\ \ff\eta\in\pp f(u),\ u\in\rrd.\end{equation}
 The mapping $\pp f:\rrd\to2^\rrd$ is the  {\it subdifferential} of $f$ and recall that it is a maximal monotone mapping (graph) in $\rrd\times\rrd$. 
 We denote by  $f'(u,v)$ the {\it directional derivative} of $f$ at $u$ in direction $v$, that is, 
 \begin{equation}\label{e1.17}
 	f'(u,v)=\lim_{\lbb\downarrow0}\frac1\lbb\,(f(u+\lbb v)-f(u)),\ \ff\,u,v\in\rrd,\end{equation}
 and note that (see, e.g., \cite{1}, p. 86)
 \begin{equation}\label{e1.18}
 	\pp f(u)=\{w\in\rrd;\ w\cdot v\le f'(u,v),\ \ff v\in\rrd.\}\end{equation}

\section{The controlled linear \FP\\ equation}\label{s2}
\setcounter{equation}{0}

We shall review here for later use a few standard results regarding the controlled \FP\ equation \eqref{e1.6a}, namely
\begin{equation}\label{e2.1}
\barr{ll}
\dd\frac{\pp\r}{\pp t}-\nu\Delta\r+\divv (u\r)=0,&\mbox{in }Q_T=(0,T)\times\rrd,\vsp 
\r(0,x)=\r_0(x),&x\in\rrd,\earr,\end{equation}
with the control input $u\in L^\9_{\rm loc}(Q_T)$ and $\r_0\in L^1$. The function $\r\in L^1_{\rm loc}(Q_T)$ is called {\it distributional solution} to \eqref{e2.1} if
\begin{equation}\label{e2.2}
\barr{r}
\dd\int_{Q_T}\r\(\frac{\pp\vf}{\pp t}+\nu\Delta\vf+u\cdot\nabla\vf\)dtdx+\int_\rrd\r_0(x)\vf(0,x)dx=0,\vsp
\ff\vf\in C^\9_0([0,T)\times\rrd).\earr\end{equation}
The function $\r$ is called a {\it variational solution} to \eqref{e2.1} if 
$$\r\in L^2(0,T;H^1)\cap W^{1,2}([0,T];H\1),$$ 
\begin{equation}\label{e2.3}
\barr{ll}
\dd\frac{d\r}{dt}-\nu\Delta\r+\divv(u\r)=0,&\mbox{ a.e. in }Q_T,\vsp 
\r(0,x)=\r_0(x),&x\in\rrd,\earr\end{equation}
where $\frac{d\r}{dt}\in L^2(0,T;H^{-1})$ is the strong derivative of $\r:[0,T]\to H\1$ or, equivalently, the $H^1$-valued distributional derivative of $\r$.

It is useful to note that if $\r\in L^2(0,T;H^1)\cap W^{1,2}([0,T];H\1)$, then $\r\in C([0,T];L^2)$ and

\begin{equation}\label{e2.4}\barr{r}
\dd\frac d{dt}\,(\r(t),\vf(t))_2=\raise-2,7mm\hbox{${}_{H\1}$}\!\!\(\frac d{dt}\,\r(t),\vf(t)\)_{\!\!H^1}
+\raise-2,7mm\hbox{${}_{H\1}$}\!\!\(\frac{d\vf}{dt}\,(t),\r(t)\)_{\!\!H^1},\vsp\mbox{ a.e. }t\in(0,T),\earr
\end{equation}
for all $\vf\in L^2(0,T;H^1)\cap W^{1,2}([0,T];H\1)$. Here,  ${}_{H\1}(\cdot,\cdot)_{H^1}$ is the duality pairing between $H^1$ and its dual space $H\1$ with $L^2$ as pivot space, that is, $H^1\subset L^2\subset H\1$ with continuous and dense embeddings. 
As regards the existence for \eqref{e2.1}, we have

\begin{proposition}\label{p1} Let $\r_0\in L^1\cap L^2$ and $u\in (L^\9(Q_T))^d$. Then \eqref{e2.1} has a unique strong solution $\r=\r^u$. Moreover, we have
\begin{eqnarray}
|\r(t)|^2_2+\int^t_0|\nabla\r(s)|^2_2ds&\le&C_T\exp(|u|^2_\9t)|\r_0|^2_2,\ \ff t\in[0,T],\label{e2.5}\\[1mm]
|\r(t)|_m&\le&C_T|\r_0|_m,\ \ff t\in[0,T],\  m=1,2.\label{e2.6}
\end{eqnarray}

If $\r_0\ge0$, a.e. in $\rrd,$ then $\r\ge0$, a.e. in $Q_T$ and if $\r_0\in\calp$, then $\r(t)\in\calp,$ $\ff t\in[0,T].$ 
If $\r_0(x)>0,$ a.e. $x\in\rrd$, and $\log\r_0\in L^1_{\rm loc}$, then
\begin{equation}\label{e2.7}
\r(t,x)>0,\ \mbox{a.e. }(t,x)\in Q_T,\ \log\r(t)\in L^1_{\rm loc},\ \ff t\in[0,T].\end{equation}
\end{proposition}

\n(Here and everywhere in the following $|u|_\9=\|u\|_{L^\9(Q_T)}.$)
\begin{proof} It is convenient to write \eqref{e2.1} (equivalently \eqref{e2.3}) as the infinite dimensional Cauchy problem
\begin{equation}\label{e2.8}
\frac{d\r}{dt}+A(t)\r=0,\ t\in(0,T);\ \ \r(0)=\r_0,\end{equation}
where $A(t):H^1\to H\1$ is the linear continuous operator
\begin{equation}\label{e2.9a}
	{}_{H\1}(A(t)y,\vf)_{H^1}=\nu\int_\rrd\nabla y(x)\cdot\nabla\vf(x)dx+\int_\rrd y(x)u(t,x)\cdot\nabla\vf(x)dx,
	\end{equation}
$\ff\vf\in H^1.$ 
Then one gets the existence and uniqueness of a solution $\r\in L^2(0,T;H^1)\cap W^{1,2}([0,T];H\1)$ for \eqref{e2.7} by a standard existence result (see, e.g., \cite{1}, p.~64). Moreover, estimate \eqref{e2.5}  is immediate by \eqref{e2.4}. If $\r_0\ge0$, a.e. in $\rrd$, then it follows $\r\ge0$, a.e. in $Q_T$ by the maximum principle. 

We also have
\begin{equation}\label{e2.9b}
\int_\rrd\rho(t,x)dx=\int_\rrd\rho_0(x)dx,\ \ff t\in[0,T],
\end{equation}
from which it follows in particular that, if $\rho_0\in\calp$, then $\rho(t)\in\calp$, $\ff t\in[0,T]$. Indeed, by \eqref{e2.4} we have
$$\barr{l}
\dd\frac d{dt}\int_\rrd\rho(t,x)\psi(x)dx+\nu\int_\rrd\nabla\rho(t,x)\cdot\nabla\psi(x)dx\vsp\qquad
-\dd\int_\rrd u(t,x)\cdot\nabla\psi(x)\rho(t,x)dx=0,\mbox{ a.e. }t\in(0,T),
\earr$$
for all $\psi\in C^\9_0(\rrd)$. This yields
$$\barr{l}
\dd\int_\rrd\rho(t,x)\psi(x)dx+\nu\int^t_0 ds\int_\rrd\nabla\rho(s,x)\cdot\nabla\psi(x)dx\vsp\qquad
-\dd\int_\rrd ds\int_\rrd u(s,x)\cdot\nabla\psi(x)\rho(t,x)dx
=\dd\int_\rrd\rho_0(x)\psi(x)dx,\ 
\ff t\in(0,T).
\earr$$
We take $\psi(x)=\psi_n(x\equiv\eta\(\frac{|x|^2}n\)$, where $\eta\in C^1(\rr)$ is such that $\eta(r)=1$ for $0\le r\le1$, $\eta(r)=0$ for $r\ge2$. Taking into account that $|\nabla\psi_n(x)|\le\frac C{\sqrt{n}}$, $\ff x\in\rrd$, and that $u\in L^\9(Q_T)$, $\nabla\rho\in L^2(Q_T)$, we get
$$\left|\int_\rrd\rho(t,x))\psi_n(x)dx-\int_\rrd\rho_0(x)dx\right|\le\frac C{\sqrt{n}},\ \ff n.$$
This yields
$$\left|\int_{[|x|\le n]}\rho(t,x))dx-\int_\rrd\rho_0(x)dx\right|\le\frac C{\sqrt{n}},\ \ff n,$$
and so, letting $n\to\9$,  \eqref{e2.9b} follows.\newpage 

Assume now that $\log \r_0\in L^1_{\rm loc}$. To prove \eqref{e2.7}, we consider an arbitrary function $\psi\in C^\9_0(\rrd)$ such that $\psi\ge0$ on $\rrd$ and take in \eqref{e2.4} $\vf=(\rho+\vp)\1\psi$. We get 
$$\barr{l}
\dd\frac d{dt}\int_\rrd\psi(x)\log(\r(t,x)+\vp)dx
-\nu\dd\int_\rrd\frac{|\nabla\r(t,x)|^2}{(\r(t,x)+\vp)^2}\,\psi(x)dx\vspace*{3mm}\\ 
\qquad=-\nu\dd\int_\rrd\frac{\nabla\r(t,x)\cdot\nabla\psi(x)}{\r(t,x)+\vp}\,dx
-\dd\int_\rrd\frac{\r(t,x)u(t,x)\cdot\nabla\r(t,x)\psi(x)}{(\r(t,x)+\vp)^2}\,dx\vspace*{3mm}\\ 
\qquad+\dd\int_\rrd\frac{\r(t,x)u(t,x)\cdot\nabla\psi(x)}{\r(t,x)+\vp}\,dx,\mbox{ a.e. }t\in(0,T).\earr$$
This yields
$$-\!\dd\int_\rrd\!\!\psi(x)\log(\r(t,x){+}\vp)\le\!-\!\int_\rrd\!\!\psi(x)\log(\r_0(x){+}\vp)dx{+}C|u|_\9(|\psi|_1{+}|\nabla\psi|_1),$$
a.e. $t\in(0,T).$ We set $\ooo_t=\left\{x\in\rrd;\rho(t,x)\ge\frac12\right\}.$ By \eqref{e2.9a} it follows that $m(\ooo_t)\le2|\rho_0|_1$, $\ff t\in[0,T]$. (Here, $m$ is the Lebesgue measure.) This implies that for $0<\vp\le\frac12$, $\log(\rho(t,x)+\vp)\le0$, $\ff x\in\ooo_t$, $t\in(0,T)$, and so, since $\log\rho_0\in L^1_{\rm loc}$, by \eqref{e2.9b}  we have on $\ooo^c_t=\rrd\setminus\ooo_t$,
$$\int_{\ooo^c_t}|\log(\rho(t,x)+\vp)|\psi(x)dx\le C,\ \ff t\in[0,T].$$
Then, by Fatou's lemma it follows that
$$\int_{\ooo_t}|\log(\rho(t,x))|\psi(x)dx\le C(\rho_0), \ \ff t\in[0,T],$$
where
\begin{equation}\label{e2.11} 
	C(\rho_0)=\int_\rrd\psi(x)|\log(\rho_0(x))|dx+C_1|u|_\9(|\psi|_1+|\nabla\psi|_1).\end{equation}
Since $\log(\rho(t,\cdot))\in L^1(\ooo_t),$ we get
\begin{equation}\label{e2.12}
\int_\rrd\psi(x)|\log\r(t,x)|dx\le C(\rho_0),\mbox{ a.e. }t\in(0,T),
\end{equation}
and, therefore, $\log\r(t,x)\in L^1_{\rm loc}$, a.e. $t\in(0,T)$, as claimed. 

\n Estimate \eqref{e2.6} follows by  \eqref{e2.5} if $m=2.$ We multiply \eqref{e2.1} by $\calx_\delta(\rho)$, where
$$\calx_\delta(r)=\left\{\barr{cll}
1&\mbox{ if }&|r|\le\delta,\vsp
\dd\frac r\delta&\mbox{ if }&|r|<\delta,\earr\right.$$
and integrate on $(0,t)\times\rrd$. Taking into account that
$$\lim_{\lbb\to0}\int^t_0ds\int_\rrd(u(s,x)\cdot\nabla\r(s,x))\r(s,x)\calx'_\delta(\r(s,x))dx=0,$$
we get  
$$\ov{\lim_{\delta\to0}}\int_\rrd\calx_\delta(\rho(t,x))
\rho(t,x)dx\le\int_\rrd|\rho_0(x)|dx,\ \ff t\ge0,$$
which yields \eqref{e2.6} for $m=1$.\end{proof}

Proposition \ref{p1} can be extended to all $\rho_0\in L^1$. Namely,

\begin{proposition}\label{p2.2} Let $\rho_0\in L^1$ and $u\in (L^\9(Q_T))^d$. Then \eqref{e2.1} has a unique distributional solution $\rho$ which satisfies
\begin{eqnarray}
&\rho\in C([0,T];L^1)
\cap L^\beta(Q_T),
\label{e2.13}\\
&|\rho(t)|_m\le C|\r_0|_1(1+|u|_\9)t^{-\frac d2\(1-\frac1m\)},\   \ff t\in(0,T),\ m\in\mbox{$\left[1,\frac d{d-1}\right)$},\label{e2.14}\\
&|\rho(t)|_1\le|\rho_0|_1,\ \ff t\in[0,T],\label{e2.14a}
\end{eqnarray}
where \begin{equation}\label{e2.15a}
	\beta\in\left[1,\frac d{d-1}\right)\mbox{ if }d\ge2;\ \beta\in[1,3)\mbox{ if }d=1.
\end{equation}
Moreover, if $\rho_0\in\calp$, then $\rho(t)\in\calp$, $\ff t\in[0,T],$ and if $\log\rho_0\in L^1_{\rm loc}$, then \eqref{e2.7} holds.
\end{proposition}

\begin{proof} Let $\{\r^n_0\}\subset L^1\cap L^2$ be such that $\r^n_0\to\r_0$ strongly in $L^1$ as $n\to\9$. By \eqref{e2.6}, where $m=1$,we get for the corresponding solution $\r_n$ to \eqref{e2.1}
	$$|\r_n(t)-\r_m(t)|_1\le|\r_0^n-\r^m_0|_1,\ \ff t\in[0,T],$$ and so $\r_n\to \r$ in $C([0,T];L^1)$, where $\r$ is a distributional solution to \eqref{e2.1} (see \eqref{e2.2}). Moreover, if $\r_0\in\calp$, the sequence $\{\r^n_0\}$ can be chosen in such a way that $\r^n_0\in\calp$ as so $\r_n(t)\in\calp$, $\ff t\in[0,T]$. Hence, $\r(t)\in\calp$, $\ff t\in[0,T]$. Assume now that $\r_0(x)>0$, a.e. $x\in\rrd$ and that $\log\r_0\in L^1_{\rm loc}$. We choose $\r^n_0=\r_0\(1+\frac 1n\,\r_0\)\1\in L^1\cap L^2.$ We have\newpage
$$\log\r^n_0=\log\r_0-\log\(1+\frac1n\,\r_0\)\le\log\r_0,\mbox{ a.e. }$$
and so, by \eqref{e2.12} it follows that
$$\int_\rrd\psi(x)\log|\r_n(t,x)|dx\le C,\ \ff n,$$and so, for $n\to\9$ it follows that $\log\r(t,x)\in L^1_{\rm loc}(\rrd)$, $\ff t>0$, as claimed. 

To prove \eqref{e2.13}--\eqref{e2.14a}  we note that by \eqref{e2.1} we have
\begin{equation}\label{e2.14aa}
\barr{ll}
\r(t,x) 
 =\dd\int_\rrd E(t,x-\xi)\r_0(\xi)d\xi\vsp
\quad+\dd\int^t_0ds\int_\rrd(\nabla_x
E(t-s,x-\xi)\cdot u(s,\xi))\r(s,\xi)d\xi,\ \ff(t,x)\in Q_T.
\earr\end{equation}
where $E$ is the heat kernel
\begin{equation}\label{e2.15}
E(t,x)=(2\nu\sqrt{\pi})^{-d}t^{-\frac d2}\exp\(-\frac{|x|^2}{4\nu^2t}\),\ t>0,\ x\in\rrd.\end{equation}
By Young's inequality we have, for all $z\in L^q$,
\begin{equation}\label{e2.16}
\hspace*{-7mm}\barr{l}	
\quad\dd\int_\rrd E(t,x-\xi)z(\xi)d\xi|_m 
 \le\oo t^{\frac d2\(\frac1m-\frac1q\)}|z|_q,
\,1\le q\le m<\9, t>0,\vsp 
\dd\int_\rrd \nabla_x E(t,x-\xi)z(\xi)d\xi|_m 
 \le\oo t^{\frac d2\(\frac1m-\frac1q\)-\frac12}|z|_q,
\,1\le q\le m<\9, t>0,\earr\hspace*{-7mm}
\end{equation}
where
\begin{equation}\label{e2.16a}
\oo=(2\nu)^{-d-2}\pi^{-\frac d2}\int_\rrd\exp\(-\frac{|y|^2}{4\nu^2}\)(1+|y|)dy.\end{equation}
Then, by \eqref{e2.14aa} we get 
\begin{equation}\label{e2.16aa}
\hspace*{-7mm}	\barr{ll}
	|\r(t)|_m\!\!\!
	&\le\oo t^{\frac d2\(\frac 1m-1\)}|\r_0|_1
	+\oo\dd\int^t_0(t-s)^{\frac d2\(\frac1m-1\)-\frac12}|\r(s)u(s)|_1ds\vsp 
	&\le\oo|\r_0|_1t^{\frac d2\(\frac1m-1\)}
	+2\alpha(d(1-m)+m)\1\oo|\r_0|_1|u|_\9 t^{\frac d2\left(\frac1m-1\right)+\frac12},\vsp&\hfill \ff t\in(0,T),
	\earr\hspace*{-7mm}\end{equation}
which implies \eqref{e2.14} for $m\in\left[1,\frac d{d-1}\right)$. 
\begin{equation}\label{e2.17}
	\|\r\|_{L^\beta(Q_T)}
	\le C|\r_0|_1(1+T^{\frac12}|u|_\9).  
	\end{equation}
Then, for each $\beta$ satisfying \eqref{e2.15a} we have $\r\in L^\beta(Q_T)$ and so \eqref{e2.13} follows. 
As regards the uniqueness of distributional $\r\in C([0,T];L^1)$ to \eqref{e2.1}, it is well known (see, e.g., \cite{3a} in a more general setting).\end{proof}


\begin{remark}\label{r2.3}\rm By \eqref{e2.14} it follows that, if $d=1$, then $\r\in L^2(Q_T)\cap L^\9(\delta,T;L^2)$, $\ff\delta\in(0,T),$ and this yields
	$$\r\in L^2(\delta,T;H^1)\cap W^{1,2}(\delta,T;H\1),\ \ff\delta\in(0,T).$$
\end{remark}

\section{The main results}\label{s3}
\setcounter{equation}{0}

We consider here the optimal control problem \eqref{e1.4}  under the following hypotheses: 
\begin{itemize}
	\item [(i)] $L\equiv L(t,x,u):Q_T\times\rrd\to\rr$ is convex,   lower-semicontinuous in $u$, measurable in $(t,x)\in Q_T$, and
	\begin{eqnarray}
	&	L(t,x,u)=+\9,\mbox{ a.e. }(t,x)\in Q_T,\mbox{ for }|u|>a,\label{e3.1}\\[1mm]
&\hspace*{-5mm}\sup\{|L(t,x,u)|;|u|\le a\}\!\in\! L^2(Q_T), L(t,x,0)\!\in\! L^\9(Q_T)\cap L^1(Q_T).\quad   \label{e3.5}
	\end{eqnarray}
	\item[(ii)] $g\equiv g(t,x,q):Q_T\times\rr\to\rr$ is convex and continuous in $q\in\rr$, measurable in $(t,x)\in Q_T$, and 
	\begin{equation}\label{e3.2a}
		\limsup_{q\to\9}\frac{g^0_q(t,x,q)q}{|g(t,x,q)|}=C<\9\mbox{ uniformly in }(t,x)\in Q_T,
		\end{equation}
		\begin{eqnarray}
	&&g(t,x,q)\ge m_1(t,x)q,\ \ff(t,x,q)\in Q_T\times\rr,\label{e3.6}\\[1mm]
	&&\sup\{|g(t,x,q)|;\,|q|\le\mu\}=m_\mu(t,x),\ (t,x)\in Q_T,\ \ff\mu\in\rr,\qquad\label{e3.7}
	\end{eqnarray}where $m_1,m_\mu\in L^1(Q_T)$, and
	$$g^0_q(t,x,q)=\sup\{\eta\in\rrd;\ \eta\in g_q(t,x,q)\},\ \ff(t,x)\in Q_T,\ q\in\rrd.$$
	\item[(iii)] $g_0\equiv g_0(x,q):\rrd\times\rr\to\rr$ is convex and continuous in $q\in\rr$, measurable in $x\in\rrd$, and
	\begin{equation}
	\label{e3.8}m_2(x)q\le g_0(x,q)\le C_1|q|^2+C_2,\ae x\in\rr;\,q\in\rr
	\end{equation}
	where $m_2\in L^1\cap L^\9$, 
	 $C_1,C_2\in\rr,\ 
	C_1>0.$  
	\end{itemize}

\n If $H=H(t,x,q)$ is the conjugate of $u\to L(t,x,u)$, that is,
\begin{equation}\label{e3.9}
H(t,x,q)=\sup\{q\cdot u-L(t,x,u);\ u\in\rrd\},\ (t,x)\in Q_T,\ q\in\rr,\end{equation}
then hypothesis (i) is equivalent to the following one
\begin{itemize}
	\item [(i)$'$] $H:Q_T\times\rr\to\rr$ is convex and continuous in $q\in\rrd$, measurable in $(t,x)\in Q_T$ and
\begin{equation}\label{e3.10}
m_3(t,x)\le H(t,x,q)\le a|q|+C_3,\ \ff q\in\rr,\ (t,x)\in Q_T,\end{equation}where $m_3\in L^\9(Q_T)\cap L^1(Q_T)$. 
\end{itemize}
In the following, we shall denote by $H_q(t,x,q)$ the subdifferential of the function $q\to H(t,x,q)$, that is,
\begin{equation}\label{e3.11}
\barr{r}
H_q(t,x,q)=\{\zeta\in\rrd;\,\zeta\cdot(q-\bar q)\ge H(t,x,q)-H(t,x,\bar q);\,\bar q\in\rrd\},\vsp
(t,x)\in Q_T,\,q\in\rrd,\earr\end{equation}
and note that by \eqref{e3.10} we have
\begin{equation}\label{e3.12}
\sup\{|\zeta|;\,\zeta\in H_q(t,x,q)\}\le C_3,\ \ff q\in\rrd,\ (t,x)\in Q_T.\end{equation}
Similarly, we shall denote by $g_q$ and $(g_0)_q$ the subdifferentials of the functions $q\to g(t,x,q)$ and $q\to g_0(x,q)$, respectively.

We set
$$\Lambda=\left\{(\alpha,\beta);\alpha\in\left[1,\frac{d+2}{d+1}\right),  \beta\in\left[1,\frac d{d-1}\right)\mbox{ if }d>1;\, \beta\in[1,3)\mbox{ if }d=1\right\}.$$
Theorem \ref{t1} below is the main result.

\begin{theorem}\label{t1} Let $\r_0\in L^1$ be such that	
\begin{equation}\label{e3.13}
\rho_0(x)>0,\ \mbox{ a.e. }  x\in\rrd;\ \log \r_0\in L^1_{\rm loc},\ g(\cdot,\cdot,\r_0)\in L^1(Q_T). \end{equation}
Then, under hypotheses {\rm(i)--(iii)} there is an optimal solution $(u^*,\r^*=\r^{u^*})$ to problem \eqref{e1.5a}, which satisfies, for all $(\alpha,\beta)\in\Lambda$ and $m_0\in\left[1,\frac d{d-1}\right)$,
\begin{eqnarray}
&u^*\in (L^\9(Q_T))^d;\,\r^*(t,x)>0,\ \ae(t,x)\in Q_T,\label{e3.14}\\
&\r^*\in C([0,T];L^1)\cap L^\beta(Q_T),\label{e3.15}\\
&u^*(t,x)\in H_q(t,x,\nabla p(t,x)),\ae(t,x)\in Q_T,\label{e3.18}\\
&p\in L^\alpha(0,T;W^{1,\alpha}(\rrd)),
\label{e3.19}\\
&\dd\frac{\pp p}{\pp t}+\nu\Delta p+H(t,x,\nabla p)=\eta(t,x)\mbox{ in }Q_T,\label{e3.20}\\
&p(T,x)=-\eta_0(x),\ x\in\rrd,\nonumber\end{eqnarray}
\begin{eqnarray}
&\eta\in L^1(Q_T),\ \eta(t,x)\in g_\r(t,x,\r(t,x)),\ae(t,x)\in Q_T,\label{e3.21}\\
&\eta_0(x)\in-(g_0)_\r(x,\r(T,x)),\ae x\in\rrd;\ \eta_0\in L^{m_0},\label{e3.22}\\
&\rho^*(t)\in \calp,\ \ff t\in(0,T)\mbox{ if }\rho_0\in\calp.\label{e3.18a}
\end{eqnarray}
\end{theorem}

\n We note that the solution $p$ to the backward Cauchy problem \eqref{e3.20} is taken in the  distributional sense \eqref{e2.2}. However, since by \eqref{e3.19} we have, for \mbox{$\frac1{\alpha^*}=1-\frac1\alpha,$}
$$\Delta p\in L^\alpha(0,T;W^{-1,\alpha^*}),\ H(t,x,\nabla p)\in L^\alpha(Q_T),\ \eta\in L^1(Q_T),$$it follows that
$$\frac{\pp p}{\pp t}\in L^\alpha(0,T;\calx),\ \calx=W^{-1,\alpha^*}+L^\alpha+L^1,$$
and, therefore, \eqref{e3.20} holds, a.e. $t\in(0,T),\ x\in\rrd.$  

As regards the uniqueness of the optimal pair $(u^*,\r^*)$ in problem \eqref{e1.5a},  we~have 

\begin{theorem}\label{t2} Let $(\alpha,\beta)\in\Lambda$, $m_0\in\left[1,\frac d{d-1}\right)$, and
	$\alpha^*=\frac\alpha{\alpha-1}$.   Assume further in Theorem {\rm\ref{t1}} that $\r_0\in L^{\alpha^*}$ and that 
	the function $q\to g(t,x,q)$ is strictly convex and	
\begin{equation}\label{e3.23}
\limsup_{q\to\9}\frac{g^0_q(t,x,q)}{|q|^{\alpha^*}}=C<\9\mbox{ uniformly in }(t,x)\in Q_T.\end{equation}
Then, there is a unique solution $u^*$ to   problem \eqref{e1.5a}  and the system
\begin{equation}\label{e3.24}
\barr{ll}
\dd\frac{\pp\r}{\pp t}-\nu\Delta\r+\divv(\r u^*)=0&\mbox{ in }Q_T,\vsp
\dd\frac{\pp p}{\pp t}+\nu\Delta p+H(t,x,\nabla p)=\eta&\mbox{ in }Q_T,\vsp 
\r(0)=\r_0,\ p(T,x)=-\eta_0,&\mbox{ in }\rrd,\earr
\end{equation}
\begin{equation}\label{e3.25}
u^*\in H_q(t,x,\nabla p),\ae\mbox{in }Q_T,\end{equation}
\begin{equation}\label{e3.26}
\eta\in g_\r(t,x,\r),\ae\mbox{in }Q_T;\ \eta_0\in(g_0)_\r(x,\r(T,x))\mbox{ in }\rrd,\end{equation}
has a unique solution $(\r,p)$ which satisfies \eqref{e3.15}--\eqref{e3.19}, namely,
\begin{eqnarray}
&\r\in C([0,T];L^1)\cap L^\beta(Q_T),\
\r>0,\mbox{ a.e. on }Q_T,\label{e3.27}\\
& p\in L^\alpha(0,T;W^{1,\alpha}(\rrd));\ \eta\in L^1(Q_T),\ \eta_0\in L^{m_0}.\label{e3.28}
\end{eqnarray} 
\end{theorem}

Theorem \ref{t1} and Theorem \ref{t2} can be reformulated in terms of the  mean field system \eqref{e1.1}, where $H$ satisfies assumption (i)$'$ and $F:Q_T\times \rr\to2^\rr$, $G:\rrd\times\rr\to2^\rr$ are the multivalued mappings
$$F(t.x.q)\equiv g_q(t,x,q),\ \ G(x,q)\equiv (g_0)_q(x,q).$$Namely, we have

\begin{theorem}\label{t3} Let $\r_0$ satisfy \eqref{e3.13}. Then, under hypotheses {\rm(i)$'$, (ii), (iii)} there is a  solution $(\r,p)$ to system \eqref{e1.1} which satisfies  	
\begin{eqnarray}
&\r\in C([0,T];L^1)\cap L^\beta(Q_T),\label{e3.29}\\
&p\in L^\alpha(0,T;W^{1,\alpha}(\rrd)).
\label{e3.30}
\end{eqnarray}
More precisely, we have
\begin{equation}\label{e3.31}
\barr{ll}
\dd\frac{\pp\r}{\pp t}-\nu\Delta\r+\divv(\r\zeta)=0&\mbox{ in }Q_T,\vsp
\dd\frac{dp}{dt}-\nu\Delta p+H(t,x,\nabla p)=\eta&\mbox{ in }Q_T,\vsp
\r(0)=\r_0,\ p(T)=-\eta_0 &\mbox{ in }\rrd,\earr
\end{equation}
where $\eta\in L^1(Q_T),\ \zeta\in (L^\9(Q_T))^d,\ \eta_0\in L^{m_0}$, and
\begin{eqnarray}
&\!\!\!\!\!\!\!\zeta(t,x)\in H_q(t,x,\!\nabla p(t,x)),\, \eta(t,x)\in F(t,x,\r(t,x)),\mbox{\,a.e.\,}(t,x)\in Q_T,\  \label{e3.32}\\
&\eta_0(x)\in G(x,\r(T,x)),\ae x\in\rrd.\label{e3.33}
\end{eqnarray}
The solution $(\r,p)$ is unique if $\r_0\in L^{\alpha^*}$ and  the map $q\to F(t,x,q)$ is strictly monotone and condition \eqref{e3.23} holds. 
\end{theorem}

\begin{remark}\label{r3.4} \rm If system \eqref{e3.31} is also sufficient for optimality in problem \eqref{e1.5a}, which is the case if $(\r,p)$ are sufficiently smooth, then the uniqueness in \eqref{e3.31} follows directly by the strong convexity of the cost functional \eqref{e1.5a} and, in particular, of $g$.
\end{remark}

We also note (see Remark \ref{r2.3}) that, if $d=1$, then the solution $\r$ to \eqref{e3.31} satisfies also 
$$\r\in L^2(Q_T)\cap W^{1,2}(\delta,T;H^1)\cap L^2(\delta,T;H^1),\ \ff\delta\in(0,T).$$
A typical example covered by the above theorem and which agrees with the mean field game model \eqref{e1.2} with control constraints is

$$L(t,x,u)=I_U(u)+L_0(t,x,u),\ u\in\rrd, \ (t,x)\in Q_T,$$where $U$ is a closed, convex and  bounded set of $\rrd$, 
$$I_U(u)=0\mbox{ if }u\in U;\ I_U(u)=+\9\mbox{ if }u\,\ov{\in}\,U,$$while $L_0:(Q_T)\times\rrd\to\rr$ is measurable in $(t,x)$, convex and continuous in $u$ and satisfies \eqref{e3.5}. In this case, 
$$\barr{rcl}
H(t,x,q)&=&\sup\{q\cdot u-L_0(t,x,u);\ u\in U\},\vsp
H_q(t,x,u)&\equiv&(N_U+(L_0)_u(t,x,\cdot)\1(u),\earr$$
where $N_U\subset\rrd$ is the normal cone to $U$. In particular, if $L_0\equiv0$, then $H(t,x,q)\equiv I^*_U(q)$ (the support function of $U$). For instance, if $U=\{u\in\rrd;$ $|u|\le a\}$, then $H(t,x,q)\equiv a|q|$, and so $H_q$ is the multivalued function
$$H_q(t,x,q)=a\,\frac q{|q|}\mbox{ if }q\ne0;\ H_q(t,x,0)=\{z\in\rrd;\,|z|\le a\}.$$  
In this case, the mean field game system \eqref{e1.1} is a parabolic system with the free boundary $\{(t,x)\in Q_T;\,\nabla p(t,x)=0\}$. 

We also note that the above existence theorem covers the case of the mean field game system \eqref{e1.1} with discontinuous and monotonically nondecreasing functions $q\to F(t,x,q)$, $r\to G(t,x,r)$, where $q\to F(t,x,q)$ is of subgradient type.  Namely, for such a pair of functions one could apply Theorem \ref{t2} by filling the jumps in discontinuity points $\{q_i\}$ and $\{r_i\}_i$. In other words, one replaces $F$ and $G$ in \eqref{e1.1} by the multivalued maximal monotone functions

$$\barr{lcl}
\wt F(t,x,q)&=&\la\barr{l}
F(t,x,q)\mbox{ if }q\ne q_i,\ i=1,2,\\
{[F(t,x,q_i-0),\ F(t,x,q_i+0)]}\mbox{ if }q=q_i,\earr\right.\mk\\
\wt G(x,r)&=&\la\barr{l}
G(x,r)\mbox{ if }r\ne r_i,\\
{[G(x,r_i-0),\ G(x,r_i+0)]}\mbox{ if }r=r_i.\earr\right.\earr$$
In this case,  the solution  $(\r,p)$ given by  Theorem \ref{t2}  can be viewed as a Filipov solution to the mean field system \eqref{e1.1} with monotone and discontinuous coupling functions $q\to F(\cdot,q)$, $r\to G(\cdot,r)$.
 
As seen in Theorem \ref{t1}, $\r\in L^\beta(Q_T)$ and   $\r\in L^\alpha(0,T;W^{1,\alpha})$,   where  $(\alpha,\beta)\in \Lambda$. For the initial value $\r_0\in L^1$, this regularity seems to be optimal because it follows directly by the singularity of the heat kernel $E$ in $t=0$ (see \eqref{e2.16}).

\section{Proof of Theorem \ref{t1}}\label{s4}
\setcounter{equation}{0}

We set
$$\barr{r}
J(u,\r)=\dd\int_{Q_T}(L(t,x,u(t,x))\r(t,x)+g(t,x,\r(t,x)))dtdx\vsp
\qquad+\dd\int_{\rrd}g_0(x,\r(T,x))dx,\ u\in L^\9(Q_T),\ \r\in L^2(Q_T),
\earr$$
and we denote by $\r^u$ 
 the solution  to \eqref{e2.1} given by Proposition \ref{p1}. We may rewrite \eqref{e3.1} as the minimization problem
\begin{equation}\label{e4.1}
{\rm Min}\{J(u,\r^u);\ u\in (L^\9(Q_T))^d\}.\end{equation}
We prove first the existence of a solution $u$ to \eqref{e4.1}. To this end, we note first that $J(u,\r^u)\not\equiv+\9$. Indeed, for $u=0$ we have, by \eqref{e2.14aa}, 
$$\r^0(t,x)=\int_\rrd E(t,x-\xi)\r_0(\xi)d\xi,\ \ff(t,x)\in Q_T,$$
and so 
$$g(t,x,\r^0(t,x))\le\int_\rrd E(t,x-\xi)g(t,\xi,\r_0(\xi))d\xi,\ \ff (t,x)\in Q_T.$$
Then, by \eqref{e3.13}, $g(t,x,\r^0)\in L^1(Q_T)$. 

Moreover, since by \eqref{e3.5}  $L(t,x,0)\in L^\9(Q_T)$ and $\r^0\in L^1(Q_T)$, we conclude that $J(0,\r^0)\in L^1(Q_T)$, as claimed.

We consider a sequence $\{u_n,\r_n\}$, where $\r_n=\r^{u_n}$ such that
\begin{equation}\label{e4.2}
\inf_u J(u,\r^u)\le J(u_n,\r_n)\le\inf_u J(u,\r^u)+\frac1n,\ \ff n\in\nn.\end{equation}
Clearly, $\{u_n\}$ is bounded in $(L^\9(Q_T))^d$, $\r_n>0$ on $Q_T$, and by Proposition \ref{p1}  $\{\r_n\}$ is bounded in $L^\beta(Q_T)$. Moreover, by \eqref{e2.14}, 
  $\{p_n(T)\}$ is bounded in $L^{m_0}$, $m_0\in\left[1,\frac d{d-1}\right)$. Hence,  we have on a subsequence, again denoted $\{n\}$, 
\begin{equation}\label{e4.3}
\barr{rcll}
u_n&\to&u^*&\mbox{ weak$^*$ in $(L^\9(Q_T))^d$}\vsp
\r_n&\to&\r^{u^*}=\r^*&\mbox{ weakly in $L^\beta(Q_T),$}\vsp 
\r_n(T)&\to&\r^*(T)&\mbox{ weakly in }L^{m_0}.\earr
\end{equation}
By \eqref{e3.5} and  \eqref{e4.3},   we have
\begin{equation}\label{e4.3aaaa}
\dd\liminf_{n\to\9}\int_{Q_T}L(t,x,u_n)\r_n\,dtdx \ge\dd\int_{Q_T}L(t,x,u^*)\r^*\,dtdx,\end{equation}because the function $u\to\int_{Q_T}L(t,x,u)\r^u\,dtdx$ is convex and, therefore, weakly lower-semicontinuous in $L^\9(Q_T)$. More precisely, if $v_n=\r_n u_n$, then
$$\barr{c}
L(t,x,u_n)\r_n\equiv L\(t,x,\dd\frac{v_n}{\r_n}\)\r_n,\qquad 
\dd\frac{\pp\r_n}{\pp t}-\nu\Delta\r_n+{\rm div}\,v_n=0\mbox{ in }Q_T.\earr$$
Since the function $(\r,v)\to L\(t,x,\frac v\r\)\r$ is  convex and $0\le v_n\le a\r_n$, a.e. in $Q_T$, we get by the  weak-lower semicontinuity of convex integrals in $L^\beta(Q_T)\times L^\9(Q_T)$ that
$$\liminf_{n\to\9}\int_{Q_T}L\(t,x,\frac{v_n}{\r_n}\)\r_n\,dtdx
\ge\dd\int_{Q_T}L\(t,x,\frac v\r\)\r\,dtdx,$$where $v=w-\dd\lim_{n\to\9}v_n$ in $L^\beta(Q_T)$, and so \eqref{e4.3aaaa} follows. 

Since the functions $\r\to g(t,x,\r)$ and $r\to g_0(x,r)$ are convex and continuous, we have by \eqref{e4.3} and the weak-lower semicontinuity of corresponding convex integrands 
$$\barr{r}
\dd\liminf_{n\to\9}\la\int_{Q_T}g(t,x,\r_n(t,x))dtdx+\int_{\rrd}g_0(x,\r_n(T,x))dx\ra\vsp
\ge\dd\int_{Q_T}g(t,x,\r^*(t,x))dtdx+\int_\rrd g_0(x,\r^*(T,x))dx,\earr$$
and so, by \eqref{e4.3aaaa} we have
$$\liminf_{n\to\9}J(u_n,\r_n)\ge J(u^*,\r^*).$$Then, by \eqref{e4.2} it follows that $J(u^*,\r^*)=\inf\limits_u J(u,\r^u)$, as claimed.

We shall prove now that the optimal pair $(u^*,\r^*)$ satisfies \eqref{e3.18}--\eqref{e3.22}.  To this purpose, we consider the adapted optimal control problem (see, e.g., \cite{1}, p.~24, for this technique)
\begin{equation}\label{e4.5}
\underset u{\rm Min}\la J_\vp(u,\r^u)+\frac12\int_{Q_T}\vf (x)\r^u(t,x)|u(t,x)-u^*(t,x)|^2dtdx\ra,\end{equation}
\begin{equation}\label{e4.6}
\barr{ll}
J_\vp(u,\r^u)\!\!\!&\equiv\dd\int_Q(L(t,x,u(t,x))\r^u(t,x)+g_\vp(t,x,\r^u(t,x)))dtdx\vsp
&+\dd\int_\rrd(g_0)_\vp(x,\r^u(T,x))dx.\earr\end{equation}
Here, $\vf\in C(\rrd)\cap L^\9\cap L^1$ is a given function such that $\vf(x)>0,$ $\ff x\in\rrd,$ and  $\r^u_\vp$ is  the solution to \eqref{e2.1}, with initial value $\r_\vp(0)=\r^\vp_0$, where $\r^\vp_0\in L^1\cap L^2$ is such that $\r^\vp_0\to\r_0$, strongly in $L^1$ as $\vp\to0$. The functions 
$g_\vp,$ $(g_0)_\vp$ are the Yosida approximations of $q\to g(\cdot,q)$ and $q\to g_0(\cdot,q)$, respectively, that~is,

\begin{eqnarray}
g_\vp(t,x,q)&=&\min\la\frac1{2\vp}|q-\theta|^2+g(t,x,\theta);\theta\in\rr\ra,\ q\in\rr,\label{e4.7}\\[1mm]
(g_0)_\vp(x,q)&=&\min\la\frac1{2\vp}|q-\theta|^2+g_0(x,\theta);\theta\in\rr\ra,\ q\in\rr.\label{e4.8}
\end{eqnarray}

We recall (see, e.g., \cite{1}, p.~98) that the functions  $q\to g_\vp(t,x,q)$ and $q\to(g_0)_\vp(x,q)$ are  differentiable on $\rrd$ and
\begin{equation}\label{e4.10}
\barr{rcl}
g_\vp(t,x,q)\!\!\!&=&\!\!\!\frac1{2\vp}|q{-}(I{+}\vp g_q(t,x,\cdot))\1q|^2{+}g(t,x,(I{+}\vp g_q(t,x,\cdot))\1q),\vsp
\hspace*{-5mm}(g_0)_\vp(x,q)\!\!\!&=&\!\!\!\frac1{2\vp}
|q{-}(I{+}\vp (g_0)_q(x,\cdot))\1q|^2{+}g_0(x,(I{+}\vp (g_0)_q(x,\cdot))\1q).\earr
\end{equation}
In the following, for simplicity we shall denote by $g'_\vp(t,x,q)$ and $(g_0)'_\vp(x,q)$ the differentials of the functions $q\to g_\vp(t,x,q)$ and $q\to(g_0)_\vp(x,q)$, respectively.

As seen above, for each $\vp>0$, problem \eqref{e4.5} has a solution $(u_\vp,\r_\vp=\r^{u_\vp})$, where  
\begin{equation}\label{e4.11}
\barr{ll}
\dd\frac{\pp\r_\vp}{\pp t}-\nu\Delta\r_\vp+\divv(u_\vp\r_\vp)=0&\mbox{in }Q_T,\vsp 
\r_\vp(0)=\r^\vp_0&\mbox{in }\rrd.\earr
\end{equation}
By Proposition \ref{p2.2} we know that $\r_\vp\in C([0,T];L^2)\cap L^2(0,T;H^1)$, $\frac{d\r_\vp}{dt}\in L^2(\delta,T;H\1)$, and
\begin{equation}\label{e4.12a}
\	|\r_\vp(t)\|_{L^\beta(Q_T)}\le C|\r^\vp_0|_1(1+T^{\frac12}|u|_\9)\le C_1|\r_0|_1. \end{equation}
(Everywhere in the following we shall use the same symbol $C$ to denote several positive constants independent of $\vp$.)

We have
$$\barr{l}
J_\vp(u_\vp+\lbb v,\r^{u_\vp+\lbb v})+\dd\frac12\int_{Q_T}\vf  \r^{u_\vp+\lbb v}|u_\vp+\lbb v-u^*|^2dtdx
\le J_\vp(u_\vp,\r_\vp)\vsp\qquad
+\dd\frac12\int_{Q_T}\vf \r_\vp|u_\vp-u^*|^2dtdx,\ff\lbb>0,v\in (L^\9(Q_T))^d\cap (L^1(Q_T))^d,\earr$$
which yields
\begin{equation}\label{e4.13}
\barr{l}
\dd\int_{Q_T}\!\!\!\Big(L'(t,x;u_\vp,v)\r_\vp
+L(t,x,u_\vp)z+g'_\vp(t,x,\r_\vp)z+\vf \r_\vp(u_\vp-u^*)\cdot v\\
\qquad +\dd\frac12\vf |u_\vp-u^*|^2z\Big)dtdx 
 +\dd\int_\rrd(g_0)'_\vp(x,\r_\vp(T,x)z(T,x)dx=0,\vsp
\hfill \ff v\in (L^\9(Q_T))^d\cap (L^1(Q_T))^d,\earr
\end{equation}
where  $L'(\cdot;u_\vp,v)$ is the directional derivative of the function $u\to L(\cdot,u)$ in direction $v$ (see \eqref{e1.17})  and $z$ is the solution to
\begin{equation}\label{e4.14}
\barr{l}
\dd\frac{\pp z}{\pp t}-\nu\Delta z+\divv(u_\vp z+v\r_\vp)=0\mbox{ in }Q_T,\vsp
z(0,x)=0,\ x\in\rrd.\earr\end{equation}
Let
$$p_\vp\in C([0,T];L^2)\cap L^1(0,T;H^1),\ \frac{dp_\vp}{dt}\in L^2(0,T;H\1)$$
be the solution to the backward Cauchy problem
\begin{equation}\label{e4.15}
\barr{l}
\dd\frac{\pp p_\vp}{\pp t}+\nu\Delta p_\vp+u_\vp\cdot\nabla p_\vp=g'_\vp(t,x,\r_\vp)+L(t,x,u_\vp)\\
\hfill+\dd\frac12\vf |u_\vp-u^*|^2\mbox{ in }Q_T,\vsp 
p_\vp(T,x)=-(g_0)'_\vp(x,\r_\vp(T,x)),\ x\in\rrd.\earr\end{equation}
(The existence of $p_\vp$  follows as in Proposition \ref{p1}.) By \eqref{e2.14a} and \eqref{e4.14} we also have
\begin{equation}\label{e4.16}
|p_\vp(T)|_k\le C|\r_\vp(T)|_k\le C|\r^\vp_0|_1\le C_1,\ \ff\vp>0,\ k=1,2.\end{equation}
We also note that, by the maximum principle,  
$p_\vp\ge0,\mbox{ a.e. in }Q_T.$

By \eqref{e4.13}--\eqref{e4.15} and \eqref{e2.4} we get 

$$\barr{r}
\dd\int_{Q_T}\r_\vp \nabla p_\vp\cdot v\ dtdx=-\int_{Q_T}v(L(t,x,u_\vp)+(g_\vp)_\r(t,x,\r_\vp))dtdx\\
+\dd\frac12\int_{Q_T}\vf|u_\vp-u^*|^2z\ dtdx,\earr$$
and recalling \eqref{e4.13} we have

$$\barr{r}
\dd\int_{Q_T}(L'(t,x;u_\vp,v)-(\nabla p_\vp+\vf (u^*-u_\vp))\cdot v)\r_\vp\,dtdx=0,\\ \ff v\in (L^\9(Q_T))^d\cap (L^1(Q_T))^d.\earr$$Then, by \eqref{e1.18} we get
$$\int_{Q_T}(\theta_\vp-(\nabla p_\vp+\vf (u^*-u_\vp))\cdot v)dtdx\le0,\ \ff v\in  (L^\9(Q_T))^d\cap (L^1(Q_T))^d,$$where
\begin{equation}\label{e4.19a}
	\theta_s(t,x)\in L_u(t,x,u_\vp)),\mbox{ a.e. }(t,x)\in Q_T.
	\end{equation}
(Clearly, by \eqref{e1.18}, $\theta_\vp$ can be chosen a measurable function in $Q_T$.)

Taking into account that, as seen in Proposition \ref{p1}, $\r_\vp(t,x)>0$, a.e. \mbox{$(t,x)\in Q_T$}, we get
\begin{equation}\label{e4.17}
\nabla p_\vp(t,x)+\vf (x)(u^*(t,x)-u_\vp(t,x))=\theta_\vp(t,x),\ae(t,x)\in Q_T.\end{equation}
and note that, since $(L_u)\1(q)=H_q(q),$ by \eqref{e4.17} we have
\begin{equation}\label{e4.19}
\barr{r}
u_\vp(t,x)\in H_q(t,x,\nabla p_\vp(t,x)+\vf (x)(u^*(t,x)-u_\vp(t,x))),\vsp\ae(t,x)\in Q_T.\earr\end{equation}
Equivalently,
\begin{equation}\label{e4.20}
\barr{r}
H(t,x,\nabla p_\vp(t,x)+\vf(x)(u^*(t,x)-u_\vp(t,x)))=u_\vp(t,x)\cdot\nabla p_\vp(t,x)\vsp 
-L(t,x,u_\vp(t,x))+\vf (t,x)(u^*(t,x)-u_\vp(t,x))u_\vp(t,x),  \vsp 
\ae(t,x)\in Q_T.\earr\end{equation}
Then, we may rewrite equation \eqref{e4.15} as
\begin{equation}\label{e4.21}
\barr{l}
\dd\frac{\pp p_\vp}{\pp t}+\nu\Delta p_\vp+H(t,x,\nabla p_\vp+\vf (u^*-u_\vp))\\
\qquad\qquad
=g'_\vp(t,x,\r_\vp)-\dd\frac12\,\vf(|u^*|^2-|u_\vp|^2)\mbox{ in }Q_T,\vsp
p_\vp(T,x)=-(g_0)'_\vp(x,\r_\vp(T,x)),\ x\in\rrd.\earr\end{equation}
Now, we shall prove that, for $\vp\to0$,
\begin{equation}\label{e4.22}
\barr{rcll}
\rho_\vp&\to&\rho^*&\mbox{ weakly in $L^\beta(Q_T),$},\vsp
u_\vp&\to&u^*&\mbox{ strongly in }(L^2_{\rm loc}(Q_T))^d.\earr\end{equation}
To this end, coming back to \eqref{e4.5}--\eqref{e4.6}, we note that
$$J_\vp(u_\vp,\r_\vp)+\frac12\int_{Q_T}\vf \r_\vp|u_\vp-u^*|^2dtdx\le J_\vp(u^*,\r^*),\ \ff\vp>0.$$
Since
\begin{equation}\label{e4.24a}
	\barr{lcll}
g_\vp(t,x,\r^*(t,x))&\to&g(t,x,\r^*(t,x)),&\ae(t,x)\in Q_T,\vsp 
(g_0)_\vp(x,\r^*(T,x))&\to&g_0(x,\r^*(T,x)),&\ae x\in\rrd, 
\earr\end{equation}
and $g_\vp\le g$,  a.e. in $Q_T$, $(g_0)_\vp\le g_0$, a.e. in $\rrd.$ We infer that
\begin{equation}\label{e4.23}
\lim_{\vp\to0}J_\vp(u^*,\r^*)=J(u^*,\r^*).\end{equation}
On the other hand, by \eqref{e4.11}  and \eqref{e2.15a}, \eqref{e2.14}, \eqref{e2.16aa} it follows 
that $\{\r_\vp\}$ is bounded in $L^\beta(Q_T)$, $\r_\vp(T)$ is bounded in $L^{m_0}$, while by \eqref{e3.1} $\{u_\vp\}$ is weak$^*$ compact in $(L^\9(Q_T))^d$. Hence, on a subsequence, again denoted $\{\vp\}\to0$, we have
\begin{equation}\label{e4.24}
\barr{rcll}
\r_\vp&\to&\wt\r&\mbox{weakly in $L^\beta(Q_T)$}\vsp 
	\r_\vp(T)&\to&\wt\r(T)&\mbox{weakly in $L^{m_0}$}\vsp 
	u_\vp&\to&\wt u&\mbox{weak$^*$ in $(L^\9(Q_T))^d$}.
	\earr
\end{equation}

Then, it follows by \eqref{e4.10} and the lower-semicontinuity of the convex integrands corresponding to the functions   $q\to g(t,x,q),$ $q\to g_0(x,q)$ in the spaces $L^1(Q_T)$ and $L^2(\rrd)$, respectively, 
\begin{equation}\label{e4.28b}
	\barr{rcl}
\dd\liminf_{\vp\to0}\dd\int_{Q_T} (g)_\vp(t,x,\r_\vp)dtdx&\ge& \dd\int_{Q_T}g(t,x,\wt\r)dtdx\vsp  
\dd\liminf_{\vp\to0}
\dd\int_{\rrd}(g_0)_\vp(x,\r_\vp(T,x))dtdx
&\ge& \dd\int_{\rrd}g_0(x,\wt\r(T,x))dtdx\earr
\end{equation}
and, therefore, by \eqref{e4.23} we have
\begin{equation}\label{e4.25}
\barr{c}
J(u^*,\r^*)\le J(\wt u,\wt\r)+\dd\limsup_{\vp\to0}\int_{Q_T}\vf\r_\vp|u_\vp-u^*|^2dtdx\vsp 
\hspace*{-6mm}\le J(u^*,\r^*)=\dd\inf_u\{J(u,\r^u)\}.\earr\end{equation}
Hence, 
$$\lim_{\vp\to0}\int_{Q_T}\vf\,\r_\vp|u_\vp-u^*|^2dx=0.$$
Now, we  set $v_\vp=\r_\vp u_\vp$ and get
\begin{eqnarray}
&\dd\int_{Q_T}L(t,x,u_\vp)\r_\vp dtdx=\int_{Q_T}L\(t,x,\frac{v_\vp}{\r_\vp}\)\r_\vp dtdx,\label{e4.25a}\\
&\dd\frac{\pp\r_\vp}{\pp t}-\nu\Delta\r_\vp+{\rm div}\,v_\vp=0\mbox{ in }Q_T.\label{e4.25aa}
\end{eqnarray}
By \eqref{e4.25}--\eqref{e4.25aa}, we have on a subsequence 
\begin{equation}
\label{e4.25aaa}\barr{c}
v_\vp\to\wt v\mbox{ weakly in }L^\beta(Q_T),\vsp\dd\frac{\pp\wt\r}{\pp t}-\nu\Delta\wt\r+{\rm div}(\wt v)=0\mbox{ in }Q_T,\earr
\end{equation}
and, since as seen earlier the functional $(\r,v)\to\int_{Q_T}L\(t,x,\frac v\r\)\r\,dtdx$ is weakly lower semicontinuous in $L^\beta(Q_T)\times L^\beta(Q_T)$, we get by \eqref{e4.25} and~\eqref{e4.25aaa}
\begin{equation}
\label{e4.25aaaa}
\barr{l}
\dd\liminf_{\vp\to0}\int_{Q_T}L(t,x,u_\vp)\r_\vp dtdx\ge\dd\int_{Q_T}L\(t,x,\frac{\wt v}{\wt\r}\)\wt\r\,dtdx\vsp\hfill
=\dd\int_{Q_T}L(t,x,\wt u)\wt\r\,dtdx,\quad\vsp\dd\frac{\pp\wt\r}{\pp t}-\nu\Delta\wt\r+{\rm div}(\wt u\,\wt\r)=0\mbox{ in }Q.\earr
\end{equation}
Hence, by \eqref{e4.25}, \eqref{e4.25a}, \eqref{e4.25aaaa}, we get
$$\r_\vp(t,x)|u_\vp(t,x)-u^*(t,x)|\to0,\mbox{ a.e. }(t,x)\in(Q_T).$$
Taking into account that by \eqref{e2.12}
$$\int_\rrd\vf(x)|\log(\r_\vp(t,x))|dx\le C(\r_0),\mbox{ a.e. }t\in(0,T),$$
we infer that
\begin{equation}\label{e4.26}
u_\vp\to u^*,\ae\mbox{in }Q_T,\end{equation}
and, since $\{u_\vp\}$ is bounded in $(L^\9(Q_T))^d$, it follows that
\begin{equation}\label{e4.27}
u_\vp\to u^*\mbox{ in }(L^m_{\rm loc}(Q_T))^d,\ \ff m\ge1.\end{equation}
Hence, $\wt u=u^*$, \eqref{e4.24} holds for $\wt\r=\r^*=\r^{u^*}$, and so \eqref{e4.22} follows.

Now, we are going to pass to limit in \eqref{e4.21} to show that there is a solution $p$ to \eqref{e3.20}--\eqref{e3.22}, which satisfies \eqref{e3.19}. To this end, we prove first

\begin{lemma}\label{l1} The set $\{h^1_\vp(t,x)=g'_\vp(t,x,p_\vp(t,x))\}_{\vp>0}$ is weakly compact in $L^1(Q_T)$.\end{lemma}

\begin{proof} We note first that $h^1_\vp\ge0$, a.e. on $Q_T$ and we denote by $g^*_\vp$ the conjugate of the function $q\to g_\vp(t,x,q)$, that is (see \eqref{e1.10}),
	$$g^*_\vp(t,x,z)=\sup\{q\cdot z-g_\vp(t,x,q);\,q\in\rr\},\ \ff(t,x)\in Q_T,\ z\in\rr.$$We have
	$$g^*_\vp(t,x,h^1_\vp(t,x))+g_\vp(t,x,\r_\vp(t,x))=\r_\vp(t,x)h^1_\vp(t,x)),\ae(t,x)\in Q_T,$$
and taking into account that, by \eqref{e3.7},
	$$\barr{r}
	g^*_\vp(t,x,z)\ge g^*(t,x,z)=\sup\{z\cdot q-g(t,x,q);\,q\in\rr\}\ge\mu|z|-m_\mu(x),\vsp \ff(t,x)\in Q_T,\,\mu>0,\,z\in\rr,\earr$$where $m_\mu\ge0$,  $m_\mu\in L^1(Q_T)$, we get
\begin{equation}\label{e4.28}
\barr{ll}
\mu|h^1_\vp(t,x)|\!\!\!&\le\r_\vp(t,x)h^1_\vp(t,x)+m_\mu(t,x)-g_\vp(t,x,\r_\vp(t,x))\vsp 
&\le\r_\vp(t,x)h^1_\vp(t,x)+m_\mu(t,x),\ \ff\mu>0,(t,x)\in Q_T.
\earr\end{equation}
Integrating \eqref{e4.28} on $Q_T$ and using equations \eqref{e4.21} and \eqref{e4.11}, we get
\begin{equation}\label{e4.31a}
\mu\dd\int_{Q_T}h^1_\vp\,dtdx
\le\dd\int_{Q_T}\alpha_\mu\ dtdx
+\dd\int_{Q_T}g'_\vp(t,x,\r_\vp(t,x))\r_\vp(t,x)dtdx.
\end{equation}
We set $J_\vp(q)=(I+\vp g_q)\1(q),\ \ff q\in\rr$, and note by \eqref{e4.7}--\eqref{e4.8} that
$$g'_\vp(q)=g_q(J_\vp(q)),\ g_\vp(q)=g(J_\vp(q))
+\frac1{2\vp}|q-J_\vp(q)|^2,\ \ff q\in\rr.$$
Then, by \eqref{e3.2a} we have for some $C$ independent of $\vp$ and $(t,x)\in Q_T$,
$$\limsup_{q\to\9}\frac{g'_\vp(t,x,q)q}{g_\vp(t,x,q)}=C<\9$$
uniformly in $\vp$ and $(t,x)$. This implies that, for each $N>0$, we have
$$\r_\vp(t,x)h^1_\vp(t,x)\le Cg_\vp(t,x,\r_\vp(t,x))\mbox{ in }Q_N=\{(t,x)\in Q_T;\,\r_\vp(t,x)\ge N\}$$and so \eqref{e4.31a} yields
 
$$\barr{rcl}
\mu\dd\int_{Q_T}h^1_\vp(t,x)dtdx
&\le&\dd\int_{Q_T}m_\mu(t,x)dtdx
+C\int_{Q_N}g_\vp(t,x,\r_\vp(t,x))dtdx\vsp&&
+CN\dd\int_{Q_T}h^1_\vp(t,x)dtdx.
\earr$$
Taking into account that by \eqref{e4.5}--\eqref{e4.23}
\begin{equation}\label{e4.31aa}
	\sup_{\vp>0}\int_{Q_T}	 g_\vp(t,x,\r_\vp(t,x))dtdx\le C<\9,
\end{equation}
we get for $N=\frac{2\mu}C$ that
\begin{equation}\label{e4.31aaa}
	\mu\int_{Q_T} h^1_\vp(t,x) dtdx\le \int_{Q_T} m_\mu(t,x) dtdx+C,\ \ff\mu>0,\end{equation}
where $C$ is independent of $\vp$ and $\mu$.

In particular, it follows by \eqref{e4.31aaa} that $\{h^1_\vp\}$ is bounded in $L^1(Q_T)$.

 Now, let $\wt Q\subset Q_T$ be any measurable subset of $Q_T$. Integrating \eqref{e4.28} on~$\wt Q$, we get
$$\mu\int_{\wt Q} h^1_\vp\,dtdx\le\int_{\wt Q} m_\mu \,dtdx+\int_{Q_T}\r_\vp h^1_\vp\,dtdx,$$
because $\r_\vp h^1_\vp\ge0$, a.e. in $Q_T$. Then, arguing as above and again using \eqref{e4.31aa}, we~get  
$$\int_{\wt Q} h^1_\vp \le\frac1\mu \int_{\wt Q} m_\mu\,dtdx+\frac C\mu,\ \ff\mu>0,\ \vp>0.$$
Since $m_\mu\in L^1(Q_T)$ and $\wt Q\subset Q_T$ is arbitrary,   this implies that $\{h^1_\vp\}$ is equi-uniformly integrable on $Q_T$ and so, by the Dunford--Pettis theorem it is weakly compact in $L^1(Q_T)$, as claimed.\end{proof}

\begin{lemma}\label{l2} For $1\le\alpha<\frac{d+2}{d+1}$  we have
\begin{equation}\label{e4.28a}
\int_{Q_T}(|p_\vp|^\alpha+|\nabla p_\vp|^\alpha)dtdx\le C,\  \ff\vp>0,\end{equation}
	where $C$ is independent of $\vp$. \end{lemma}

\begin{proof} We set 		
	$$\barr{ll}
	\wt p_\vp(t,.x)\equiv p_\vp(T-t,x),& \wt\r_\vp(t,x)\equiv\r_\vp(T-t,x),\vsp \wt u_\vp(t,x)\equiv u_\vp(T-t,x),& \wt u^*(t,x)\equiv u^*(T-t,x),\earr	$$and rewrite \eqref{e4.21} as the forward Cauchy problem
\begin{equation}\label{e4.29}
\barr{l}
\dd\frac{\pp\wt p_\vp}{\pp t}-\nu\Delta\wt p_\vp-H (T-t,x,\nabla\wt p_\vp+\vf(\wt u^*-\wt u_\vp))\\
\hspace*{10mm}=-g'_\vp(T-t,x,\wt\r_\vp)+\dd\frac12\vf(|\wt u^*|^2-|\wt u_\vp|^2)\mbox{ in }Q_T,\vsp\wt p_\vp(0,x)=-(g_0)'_\vp(x,\wt\r_\vp(0,x)),\ x\in\rrd.\earr\end{equation}
By \eqref{e2.14aa} we may represent $\wt p_\vp$ as
\begin{equation}\label{e4.30}
\hspace*{-6mm}\barr{ll}
\wt p_\vp(t,x)\!~\!\!&
=\dd\int_\rrd\!\! E(t,x-\xi)  \wt p_\vp(0,\xi)d\xi+  \dd\int^t_0\!\!\int_\rrd E(t-s,x-\xi) (h_\vp(s,\xi)\vsp 
&+H (T-s,\xi,\nabla\wt p_\vp(s,\xi)+\vf (\xi)(\wt u^*(s,\xi)-\wt u_\vp(s,\xi)))d\xi ds,\earr\end{equation}
where $E$ is the heat kernel \eqref{e2.15} and
$$h_\vp(t,x)\equiv-g'_\vp(T-t,x,\wt\r_\vp(t,x))+\frac12\,\vf(x)
(|\wt u^*(t,x)|^2-|\wt u_\vp(t,x)|^2).$$
By \eqref{e4.30}, we get
\begin{equation}\label{e4.31}
\barr{ll}
\nabla\wt p_\vp(t,x)= \dd\int_\rrd \nabla_x E(t,x-\xi)
\wt p_\vp(0,\xi)d\xi\vsp 
\qquad+\dd\int^t_0 \int_\rrd \nabla _x E(t-s,x-\xi)  (h_\vp(s,\xi))\vsp 
\qquad+H(T-s,\xi,\nabla\wt p_\vp(s,\xi)
+\vf (\xi)(\wt u^*(s,\xi)-\wt u_\vp(s,\xi)))d\xi ds,\vsp 
\hfill\ff(t,x)\in  Q_T.\earr \end{equation} 
As seen earlier,   $p_\vp\in L^2(0,T;H^1)\cap L^\9(0,T;L^2)$ and, by \eqref{e4.16},
$|\wt p_\vp(0)|_k\le C|\r_0|_1,\ k=1,2,\ \ff\vp>0,$  and, therefore, \eqref{e4.30}--\eqref{e4.31} are well defined. However, we cannot get directly the estimate \eqref{e4.28a} from \eqref{e4.30}--\eqref{e4.31}   because {\it a priori} we do not know whether $\wt p_\vp\in L^\alpha(Q_T)$ or $\nabla \wt p_\vp\in L^\alpha(0,T;(L^\alpha)^d)=(L^\alpha(Q_T))^d$. To show this, we shall employ a fixed point argument.   Namely, we write \eqref{e4.31}~as 
\begin{equation}\label{e4.33}
\nabla\wt p_\vp=\Gamma(\nabla\wt p_\vp)+w_\vp\mbox{\ \ in }Q_T,\end{equation}
where 
\begin{equation}\label{e4.34}
\barr{lcl}
 w_\vp(t,x)&=&\dd\int_\rrd
\nabla_x E(t,x-\xi)  \wt p_\vp(0,\xi)d\xi\vsp
&&+\dd\int^t_0  \int_\rrd \nabla_x E(t-s,x-\xi)  (h_\vp(s,\xi)\vsp
&&+\vf(\xi)(|\wt u^*(s,\xi)|^2-|\wt u_\vp(s,\xi)|^2))dsd\xi,\ (t,x)\in Q_T,\earr 
\end{equation}
and $\Gamma:(L^2(Q_T))^d\to (L^2(Q_T))^d$ is defined by
\begin{equation}\label{e4.35}
\barr{ll}
\Gamma(z)(t,x)\!\!\!&=\dd\int^t_0  \int_\rrd \nabla_x E(t-s,x-\xi) (H(T-s,\xi,z(s,\xi))\vsp 
&+\,\vf(\xi)(\wt u^*(s,\xi)-\wt u_\vp(s,\xi)))d\xi ds,\,\ff z\in (L^2(Q_T))^d.\earr
 \end{equation}
By \eqref{e2.16} we have 
\begin{equation}\label{e4.36}
|w_\vp(t)|_\alpha\!\!\!
\le Ct^{\frac d2\(\frac1\alpha-1\)}|\wt p_\vp(0)|_1+C\dd\int^t_0(t-s)^{\frac d2\(\frac1\alpha-1\)-\frac12}
(|h_\vp(s)|_1+1)ds. \end{equation}
Recalling that $|\wt p_\vp(0)|_1\le C|\r_0|_1$ and, by Lemma \ref{l1}, $\{h_\vp\}$ is bounded in $L^1(Q_T)$, we get 
by \eqref{e4.36} that
\begin{equation}\label{e4.37}
\|w_\vp\|_{L^\alpha(Q_T)}=\(\int^T_0|w_\vp(t)|^\alpha_\alpha dx\)^{\frac1\alpha}\le C|\r_0|_1T^{\frac1{2\alpha}(d+\alpha(1-d))}.\end{equation}
Since, by \eqref{e3.12}, $q\to H(\cdot,q)$ is Lipschitz continuous, we get by \eqref{e4.35} via equality \eqref{e2.16}
$$|\Gamma(z)(t)-\Gamma(\bar z)(t)|_\alpha
\le C\int^t_0(t-s)^{-\frac12}|z(s)-\bar z(s)|_\alpha ds,\ \ff t\in(0,T),$$
and this yields via Young's inequality
\begin{equation}\label{e4.38} 
\|\Gamma(z)-\Gamma(\bar z)\|_{L^\alpha(Q_T)}
\le C T^{\frac12}\|z-\bar z\|_{L^\alpha(Q_T)},\ \ff z,\bar z\in (L^\alpha(Q_T))^d.\end{equation}
(Here and everywhere in the following, we shall denote by $C$ several positive constants independent of $\vp,$ $T$.) Hence, for $0<T<T^*=C^{-2}$, the equation $z=\Gamma(z)+w_\vp$ has a unique solution $z_\vp\in (L^\alpha(Q_T))^d$ which, by virtue of \eqref{e4.33}, is just $\nabla\wt p_\vp$. Moreover, by \eqref{e4.37}--\eqref{e4.38} we see that
$$\barr{ll}
\|z_\vp\|_{L^\alpha(Q_T)}\!\!\!
&\le\|\Gamma(z_\vp))\|_{L^\alpha(Q_T)}+\|w_\vp\|_{L^\alpha(Q_T)}\vsp
&\le CT^{\frac12}\|z_\vp\|_{L^2(Q_T)}+C|\r_0|_1 T^{\frac1{2\alpha}}(d+\alpha(1-d)).\earr$$
This yields
\begin{equation}\label{e4.39} 
\|\nabla\wt p_\vp\|_{L^\alpha(Q_T)}
\le C|\r_0|_1 T^{\frac1{2\alpha}(d+\alpha(1-d))}
 (1-CT^{\frac12})\1,\, \ff T\in(0,T^*),\end{equation}
while, by \eqref{e4.30}, we have
\begin{equation}\label{e4.40}
|\wt p_\vp(t)|_\alpha\le C|\wt p_\vp(0)|_1 t^{-\frac 12\(d+1-\frac d\alpha\)}+C\dd\int^t_0(t-s)^{-\frac 12\(d+1-\frac d\alpha\)}|\nabla\wt p_\vp(s)|_\alpha ds,\end{equation}
and so, by \eqref{e4.39} this yields
\begin{equation}\label{e4.41}
\|\wt p_\vp\|_{L^\alpha(Q_T)}\le C|\r_0|_1 T^{\frac1{2\alpha}(d+\alpha(1-d))},\ \ff\vp>0,\ \ff T\in(0,T^*).\end{equation}
Moreover, by \eqref{e4.39} and \eqref{e4.40}, we get also, for $\alpha=1$,
\begin{equation}\label{e4.42}
\barr{ll}
|\wt p_\vp(t)|_1\!\!\!&\le C|\wt p_\vp(0)|_1+C\dd\int^t_0|\nabla\wt p_\vp(s)|_1ds\vsp
&\le C|\r_0|_1(1+T^{\frac12}(1-CT^{\frac12})\1,\ t\in(0,T^*).\earr\end{equation}
Then, iterating the above fixed point argument,  
one infers that \eqref{e4.39}--\eqref{e4.41} hold for all $T>0$, as claimed.  
\end{proof}

\n{\it Proof of Theorem {\rm\ref{t1}} $($continued$)$.} Since $\{p_\vp\}$ and $\{\nabla p_\vp\}$ are bounded in $(L^\alpha(Q_T))^d$, $(L^\alpha(Q_T))^d$, respectively, it follows that $\{H(t,x,\nabla p_\vp)\}$ is bounded in $L^\alpha(Q_T)$ and so, on a subsequence $\{\vp\}\to0$, we have, for any $\alpha\in\left[1,\frac{d}{d-1}\)$,
\begin{equation}\label{e4.43}
\barr{rcll}
p_\vp&\to&p&\mbox{weakly in $L^\alpha(Q_T)$}\vsp 
\nabla p_\vp&\to&\nabla p&\mbox{weakly in $(L^\alpha(Q_T))^d$}\vsp 
H(t,x,\nabla p_\vp)&\to&\zeta^*&\mbox{weakly in $L^\alpha(Q_T)$}.\earr\end{equation}
By Lemma \ref{l2} we have
\begin{equation}\label{e4.44}
\barr{lcll}
g'_\vp(t,x,\r_\vp)&\to&\eta&\mbox{ weakly in }L^1(Q_T),\earr\end{equation}
while, by \eqref{e3.8} and \eqref{e4.8},  it follows that $\{(g_0)'_\vp(\cdot,\r_\vp(T))\}$ is bounded in $L^2$ and so, selecting further a subsequence $\{\vp\}\to0$, we have
\begin{equation}\label{e4.45}
\barr{lcll}
(g_0)'_\vp(\cdot,\r_\vp(T))&\to&\eta_0&\mbox{ weakly in }L^2.\earr\end{equation}
To prove \eqref{e3.22}, we note that

$$\barr{l}
\dd\int_\rrd(g_0)'_\vp(x,\r_\vp(T,x))(\r_\vp(T,x)-v(x))dx\vsp
\qquad\ge\dd\int_\rrd((g_0)_\vp(x,\r_\vp(T,x))-(g_0)_\vp(v(x)))dx,\ \ff v\in L^2,\earr$$
which, by \eqref{e4.28b} yields
$$\barr{l}
\dd\liminf_{\vp\to0}\int_\rrd(g_0)'_\vp(x,\r_\vp(T,x))\r_\vp(T,x)dx\vsp
\quad\ge\dd\int_\rrd\eta_0(x)v(x)dx+\int_\rrd(g_0(x,\r(T,x))-g_0(x,v(x)))dx,\ \ff v\in L^2,\earr$$
and, therefore, $\eta_0(x)\in g'_0(x,\r(T,x)),$ a.e. $x\in\rrd$, as claimed. 
On the other hand, we have
\begin{equation}\label{e4.46}
\lim_{\vp\to0}|\r_\vp(t)-\r^*(t)|_1=0,\ \ \ff t\in[0,T].\end{equation}
Indeed, by \eqref{e4.11} and \eqref{e2.3}   where $u=u^*$, $\rho=\rho^*$, we have
$$\dd\frac\pp{\pp t}(\r_\vp-\r^*)-\nu\Delta(\r_\vp-\r^*)+\divv((u_\vp-u^*)\r^*)+\divv(u_\vp(\r^*-\r_\vp))=0\mbox{ in }Q_T,$$
and by \eqref{e2.16} we get 
$$\barr{ll}
|\r_\vp(t)-\r^*(t)|_1\!\!\!
&\ge\oo|\r^\vp_0-\r_0|_1+\oo\dd\int^t_0(t-s)^{-\frac12}(|u_\vp(s)-u^*(s))\r^*(s)|_1\vsp
&+a|\r_\vp(s)-\r^*(s)|_1)ds,\ \ff t\in[0,T].\earr$$
This yields
$$\|\r_\vp-\r^*\|_{L^1(0,t;L^1)}
\le\oo|\r^\vp_0-\r_0|_1(1-\oo t^{\frac12}\delta(\vp))\1,\ \ff t\in[0,T],$$
where $\delta(\vp)=\|(u_\vp-u^*)\r^*\|_{L^1(0,T;L^1)}.$ Since $\delta(\vp)\to0$ and $|\r^\vp_0-\r_0|_1\to0$ as $\vp\to0$, we get
\begin{equation}\label{e4.53a}
\lim_{\vp\to0}\|\r_\vp-\r^*\|_{L^1(0,t;L^1)}=0,\ \ff t\in(0,(\oo\delta(\vp))^{-2}).
\end{equation}
Taking into account that $|\r_\vp(t)|_1\le|\r^\vp_0|_1\le C|\r_0|_1$, $\ff t\in[0,T]$, we may iterate \eqref{e4.53a} and get so \eqref{e4.46}, as claimed.

Then, by \eqref{e4.45} and \eqref{e4.46} it follows \eqref{e3.33}. We also note that
\begin{equation}\label{e4.47}
u^*(t,x)\in H_q(t,x,\nabla p(t,x)),\ae(t,x)\in Q_T.\end{equation}
Indeed, by \eqref{e4.19} we have
\begin{equation}\label{e4.48}
\barr{l}
\dd\int_{\wt Q}u_\vp\cdot(\nabla p_\vp+\vf(u^*-u_\vp)-\theta)dtdx\\\qquad
\ge\dd\int_{\wt Q}(H(t,x,\nabla p_\vp+\vf(u^*-u_\vp))-H_\vp(t,x,\theta))dtdx,\\\hfill \ff\theta\in(L^\9(\wt Q))^d\cap(L^1(\wt Q))^d,\earr\end{equation}
where $\wt Q\subset Q_T$ is an arbitrary measurable set. 

Since the functional $z\to\int_{\wt Q}H(t,x,z)dxdt$ is weakly lower-semicontinuous in $(L^\alpha(\wt Q))^d$, we have by \eqref{e4.43} and \eqref{e4.19}

$$\liminf_{\vp\to0}\int_{\wt Q}H (t,x,\nabla p_\vp+\vf(u^*-u_\vp))dtdx\ge\int_{\wt Q}H(t,x,\nabla p)dtdx$$and so, letting $\vp\to0$ in \eqref{e4.48}, we get
$$\barr{r}\int_{\wt Q} (u^*-\theta)\cdot\nabla p\,dtdx\ge \int_{\wt Q}(H(t,x,\nabla p)-H(t,x,\theta))dtdx,\vsp \ff\theta\in (L^\9(\wt Q))^d\cap(L^1(\wt Q))^d.\earr$$
Since $\wt Q\subset Q_T$ is arbitrary, this implies \eqref{e4.47}, as claimed.

Let us prove now that
\begin{equation}\label{e4.49}
\zeta^*(t,x)=H(t,x,\nabla p(t,x)),\ae(t,x)\in Q_T.\end{equation}
Since, by \eqref{e3.12}, the function $q\to H(t,x,q)$ is Lipschitz-continuous (uniformly in $(t,x)\in Q_T$, by \eqref{e4.20}, it follows that
\begin{equation}\label{e4.50}
\hspace*{-3mm}\barr{r}
L(t,x,u_\vp(t,x))=u_\vp(t,x)\cdot\nabla p_\vp(t,x)-H(t,x,\nabla p_\vp(t,x))\vsp
\le C\vf(x)|u_\vp(t,x)-u^*(t,x)|,\ae(t,x)\in Q_T.\earr\end{equation}
As $\{H(t,x,\nabla p_\vp)\}$ and $\{\nabla p_\vp\}$ are bounded in $L^\alpha(Q_T)$ and $(L^\alpha(Q_T))^d$, res\-pec\-ti\-vely, we infer that $\{L(t,x,u_\vp(t,x))\}$ is bounded in $L^\alpha(Q_T)$, too. Moreover, by hypothesis (i), \eqref{e3.5}, it follows that the function $u\to L(t,x,u)$ is continuous  on $\{u\in \rrd;\,|u|\le a\}$, a.e. $(t,x)\in Q_T$. Then, by \eqref{e4.27}  it follows~that 
$$\lim_{\vp\to0}L(t,x,u_\vp(t,x))=L(t,x,u^*(t,x)),\ae(t,x)\in Q_T,$$and that $\{L(t,x,u_\vp)\}$ is weakly compact in $L^2(Q_T)$. We have, therefore, for $\vp\to0$, 
$$L(t,x,u_\vp)\to L(t,x,u^*)\mbox{ weakly in }L^2(Q_T).$$\newpage 
\n Hence, by \eqref{e4.43}   it follows that
$$L(t,x,u_\vp)-\nabla p_\vp\cdot u_\vp\to L(t,x,u^*)-\nabla p\cdot u^*\mbox{ weakly in }L^\alpha(Q_T),$$
and so, by \eqref{e4.43} and \eqref{e4.50} it follows that
$$L(t,x,u^*)-\nabla p\cdot u^*=-\zeta^*,\ae\mbox{in }Q_T.$$
Recalling (see\eqref{e2.11}) that \eqref{e4.47} is equivalent to
$$u^*\cdot\nabla p-L(t,x,u^*)=H(t,x,\nabla p),\ae\mbox{in }Q_T),$$we get \eqref{e4.49}, as claimed. 
Finally, letting $\vp\to0$ in the inequality
$$\barr{l}
\dd\int_{\wt Q}g'_\vp(t,x,\r_\vp(t,x))(\r_\vp(t,x)-\theta(t,x))dtdx\vspace*{-2mm}\\\qquad
\ge\dd\int_{\wt Q}(g_\vp(t,x,\r_\vp(t,x))-g_\vp(t,x,\theta(t,x)))dtdx,\ \ff\theta\in L^\9(\wt Q),\earr$$where $\wt Q$ is an arbitrary bounded and Lebesgue measurable set of $Q_T$, we get by \eqref{e4.44} and \eqref{e4.24a} that
$$\int_{\wt Q}\eta(\r^*-\theta)dtdx\ge\int_{\wt Q}(g(t,x,\r^*)-g(t,x,\theta))dtdx,\ \ff\theta\in L^\9(\wt Q),$$
because, as seen earlier, $\r_\vp\to\r^*$ in $L^1(Q_T)$  as $\vp\to0$ and so, by Egorov's theorem, we may choose $\wt Q$ in a such a way that $\r^*\in L^\9(\wt Q)$ and $\r_\vp\to\r^*$ in $L^\9(\wt Q)$. This yields
$$\eta(t,x)\in g_\r(t,x,\r^*(t,x)),\ae(t,x)\in\wt Q,$$and so the latter extends to all of $(t,x)\in Q_T$.
Then, letting $\vp\to0$ in \eqref{e4.11}, \eqref{e4.15}, it follows by \eqref{e4.43}--\eqref{e4.47} that $\r^*$, $p$ and $u^*$ satisfy conditions \eqref{e3.14}--\eqref{e3.22} in Theorem \ref{t1}.\hfill$\Box$

\section{Proof of Theorem \ref{t2}}\label{s5}
\setcounter{equation}{0}

Let $\r_0\in L^1\cap L^{\alpha^*}$ be such that  \eqref{e3.13} holds and let $\r_1,\r_2$ two solutions to system \eqref{e3.27}, which satisfy \eqref{e3.25}--\eqref{e3.28}. Then, we have

\begin{equation}\label{e5.1}
\barr{ll}
\dd\frac\pp{\pp t}\,(\r_1-\r_2)-\nu\Delta(\r_1-\r_2)
+\divv(\r_1\zeta_1-\r_2\zeta_2)=0 \mbox{ in }Q_T,\vsp
\dd\frac\pp{\pp t}\,(p_1-p_2)+\nu\Delta(p_1-p_2)
+H(t,x,\nabla p_1)-H(t,x\nabla p_2)\\
\hfill=\eta_1-\eta_2 \mbox{ in }Q_T,\vsp 
(\r_1-\r_2)(0,x)\equiv0,\   (p_1-p_2)(T,x)\equiv-(\eta^1_0(x)-\eta^2_0(x))\mbox{ in }\rrd,\earr\end{equation}
where

\begin{equation}\label{e5.2}
\barr{r}
\zeta_i\in H_q(t,x,\nabla p_i),\eta_i\in g_\r(t,x,\r_i),\eta^i_0\in(g_0)_\r(x,\r_i(T)),\vsp
 i=1,2,\ae\mbox{in }Q_T,\earr
\end{equation}
\begin{equation}
\zeta_i\in (L^\alpha(Q_T))^d,\ \eta_i\in L^1(Q_T),\ \eta^i_0\in L^2,\ i=1,2.\label{e5.3}
\end{equation}

We set $\r=\r_1-\r_2,\ p=p_1-p_2,\ \eta\equiv\eta_1-\eta_2,\ \eta_0\equiv\eta^1_0-\eta^2_0.$ 

Let $\Phi_\vp=(I-\vp\Delta)\1:L^2\to L^2$, that is, $\Phi_\vp(z)\equiv z_\vp,$ where $z_\vp-\vp\Delta z_\vp=z$ in $\rrd$. 
We note that $\Phi_\vp$ extends to allo $L^q$ and $L(L^q,L^q),\ \ff q\in [1,\9)$, and that
\begin{equation}\label{e5.4}
\lim_{\vp\to0}|\Phi_\vp(z)-z|^2_2=0.\end{equation}
Applying the operator $\Phi_\vp$ to both equations in \eqref{e5.1}, we get
\begin{equation}\label{e5.5}
\hspace*{-7mm}\barr{l}
\dd\frac\pp{\pp t}\,\Phi_\vp(\r)-\nu\Delta\Phi_\vp(\r)+\divv(\Phi_\vp(\r_1\zeta_1-\r_2\zeta_2))=0 \mbox{ in }Q_T,\vsp
\dd\frac\pp{\pp t}\,\Phi_\vp(p)+\nu\Delta\Phi_\vp(p)+\Phi_\vp(H(t,x,\nabla p_1)-H(t,x\nabla p_2))=\Phi_\vp(\eta) \mbox{ in }Q_T,\vsp 
\Phi_\vp(\r)(0)\equiv0,\ \ \Phi_\vp(p)(T)\equiv-\Phi_\vp(\eta_0) \mbox{ in }\rrd.\earr\hspace*{-7mm}\end{equation}
Since $\r_i(0)\in L^{\alpha^*}$, we have by Proposition \ref{p1} that $\r_i\in L^{\alpha^*}(Q_T)$, $i=1,2$, and since $\nabla p_i\in (L^\alpha(Q_T))^d$, $i=1,2$, it follows that
\begin{equation}\label{e5.6}
\r_i\nabla p_j\in L^1(Q_T),\ \ff i,j=1,2.\end{equation}
Moreover, by \eqref{e3.2a} and  \eqref{e3.26} we have
$$0\le\eta\le C(|\r|^{\alpha^*-1}+1),\ae\mbox{in }Q_T,$$and so $\eta\r\in L^1(Q_T)$. We also note that
\begin{equation}\label{e5.7}
\Phi_\vp(\r)\in L^{\alpha^*}(0,T;W^{2,\alpha^*}),\ \Phi_\vp(p)\in L^\alpha(0,T;W^{2,\alpha})\end{equation}
and by \eqref{e5.5}
$$\frac d{dt}\,\Phi_\vp(\r)\in L^{\alpha^*}(Q_T).$$
This implies that, for each $\vp>0,$ $t\to(\Phi_\vp(\r(t),p(t))_2$ is absolutely continuous on $[0,T]$ and so \eqref{e5.5} yields
\begin{equation}\label{e5.8}
\hspace*{-7mm}\barr{l}
(\Phi_\vp(\r(T)),p(T))_2-(\Phi_\vp(\r_0),p(0))_2
-\dd\int_{Q_T}(\r_1\zeta_1-\r_2\zeta_2)\cdot\nabla\Phi_\vp(\r)dtdx\vsp\qquad+\dd\int_{Q_T}\Phi_\vp(\r)(H(t,x,\nabla p_1)-H(t,x,\nabla p_2))dtdx=\dd\int_{Q_T}\Phi_\vp(\r)\eta\,dtdx.
\earr\hspace*{-7mm}\end{equation}
We have, for $\vp\to0$,
$$\barr{rcll}
\Phi_\vp(\r)&\to&\r&\mbox{in }L^{\alpha^*}(Q_T),\earr\ \ \ $$
and, since $\{\nabla\Phi_\vp(\rho)=\Phi_\vp(\nabla\rho)\}$ is bounded in $(L^\alpha(Q_d)^d$, we also have
$$\barr{rcll}
\nabla\Phi_\vp(p)&\to&\nabla p&\mbox{in }(L^\alpha(Q_T))^d.\earr$$
Then, letting $\vp\to0$ in \eqref{e5.8}, we get
\begin{equation}\label{e5.9}
\barr{l}
(\r(T),\eta_0)_2-\dd\int_{Q_T}(\r_1\zeta_1-\r_2\zeta_2)\cdot\nabla p\ dtdx\vsp
\qquad+\dd\int_{Q_T}\r(H(t,x,p_1)-H(t,x,\nabla p_2))dtdx =\int_{Q_T}\r\eta\ dtdx.\earr
\end{equation}
Taking into account that by \eqref{e5.2} we have
$$\zeta_i\cdot(\nabla p_i-z)\ge H(t,x,\nabla p_i)-H(t,x,z),\ \ff z\in\rr,$$and   that $\r_i\ge0$, we get by \eqref{e5.9} that
$$\int_{Q_T}\r\eta\ dtdx\le0.$$Hence, $\r\eta=0$,\ae in $Q_T$. Since $q\to g_\r(t,x,q)$ is strictly montone,  $\r=\r_1-\r_2$ and $\eta\in\pp_g(\r_1)-\pp g(\r_2)$, we get that $\r_1-\r_2=0$,\ae in $Q_T$, as claimed.\hfill$\Box$


\end{document}